\documentclass[12pt]{amsart}
\usepackage{hyperref}
 \usepackage[margin=1in]{geometry} 
\usepackage{amsmath,amsthm,amssymb,amsfonts, esint}
\usepackage{comment}
\usepackage{mathtools}
\usepackage{dsfont}

\newcommand{\R}{\mathbb{R}}
\newtheorem{theorem}[]{Theorem}
\DeclareFontFamily{OMX}{yhex}{}
\DeclareFontShape{OMX}{yhex}{m}{n}{<->yhcmex10}{}
\DeclareSymbolFont{yhlargesymbols}{OMX}{yhex}{m}{n}
\DeclareMathAccent{\wideparen}{\mathord}{yhlargesymbols}{"F3}
\newtheorem{proposition}{Proposition}[section]
\newtheorem{lemma}[proposition]{Lemma}
\newtheorem{conjecture}{Conjecture}[]

\newtheorem{corollary}[proposition]{Corollary}

\theoremstyle{definition}
\newtheorem{definition}{Definition}[section]
\theoremstyle{remark}
\newtheorem{remark}{Remark}[]
\newtheorem{example}{Example}[]

\usepackage{tikz}
\usetikzlibrary{calc}
\usetikzlibrary{calc, patterns, intersections, arrows.meta, shadings}

\usepackage{float}

\newcommand\numeq[1]
  {\stackrel{\scriptscriptstyle(\mkern-1.5mu#1\mkern-1.5mu)}{=}}
  \newcommand\numstrict[1]
  {\stackrel{\scriptscriptstyle(\mkern-1.5mu#1\mkern-1.5mu)}{<}}
  \newcommand\numineq[1]
  {\stackrel{\scriptscriptstyle(\mkern-1.5mu#1\mkern-1.5mu)}{\leq}}

\newcommand{\sint}{\int_{\mathbb{S}^{n-1}}}

\definecolor{fuchsia}{RGB}{240,50,180}
\definecolor{spring}{RGB}{180,220,150}
\definecolor{lavender}{RGB}{200,140,220}
\definecolor{raddish}{RGB}{250,120,120}
\newcommand\polar[2]{({#1*cos(#2)},{#1*sin(#2)})}

\usepackage{caption}
\usepackage{subcaption}

\definecolor{backshade}{RGB}{150, 210, 225}
\definecolor{frontshade}{RGB}{150, 220, 240}
\definecolor{transparentprple}{RGB}{190, 172, 223}
\definecolor{transparentgreen}{RGB}{190, 216, 192}

    \usetikzlibrary{positioning}

\newcommand\pgfmathsinandcos[3]{%
  \pgfmathsetmacro#1{sin(#3)}%
  \pgfmathsetmacro#2{cos(#3)}%
}
\newcommand\LongitudePlane[3][current plane]{%
  \pgfmathsinandcos\sinEl\cosEl{#2} 
  \pgfmathsinandcos\sint\cost{#3} 
  \tikzset{#1/.style={cm={\cost,\sint*\sinEl,0,\cosEl,(0,0)}}}
}

\newcommand\LatitudePlane[3][current plane]{%
  \pgfmathsinandcos\sinEl\cosEl{#2} 
  \pgfmathsinandcos\sint\cost{#3} 
  \pgfmathsetmacro\yshift{\cosEl*\sint}
  \tikzset{#1/.style={cm={\cost,0,0,\cost*\sinEl,(0,\yshift)}}} %
}

\newcommand\DrawLongitudeCirclearcNorthHemisphereONE[2][1]{
  \LongitudePlane{\angEl}{#2}
  \tikzset{current plane/.prefix style={scale=#1}}
  \pgfmathsetmacro\angVis{atan(sin(#2)*cos(\angEl)/sin(\angEl))} %
  \draw[current plane,very thick,transparentprple] (90:1) arc (90:180-19:1);
}
\newcommand\DrawLongitudeCirclearcSouthHemisphereONE[2][1]{
  \LongitudePlane{\angEl}{#2}
  \tikzset{current plane/.prefix style={scale=#1}}
  \pgfmathsetmacro\angVis{atan(sin(#2)*cos(\angEl)/sin(\angEl))} %
  \draw[current plane,very thick,transparentgreen] (-19:1) arc (-19:-90:1);
}
\newcommand\DrawLongitudeCirclearcSouthHemispherebacksideONE[2][1]{
  \LongitudePlane{\angEl}{#2}
  \tikzset{current plane/.prefix style={scale=#1}}
  \pgfmathsetmacro\angVis{atan(sin(#2)*cos(\angEl)/sin(\angEl))} %
  \draw[current plane,very thick,densely dashed,transparentgreen] (-19:1) arc (-19:-50:1);
}

\newcommand\DrawLatitudeCircle[2][1]{
  \LatitudePlane{\angEl}{#2}
  \tikzset{current plane/.prefix style={scale=#1}}
  \pgfmathsetmacro\sinVis{sin(#2)/cos(#2)*sin(\angEl)/cos(\angEl)}
  \pgfmathsetmacro\angVis{asin(min(1,max(\sinVis,-1)))}
  \draw[current plane,thin,black] (\angVis:1) arc (\angVis:-\angVis-180:1);
  \draw[current plane,thin,dashed] (180-\angVis:1) arc (180-\angVis:\angVis:1);
}

\newcommand\DrawLatitudeCirclethick[2][1]{
  \LatitudePlane{\angEl}{#2}
  \tikzset{current plane/.prefix style={scale=#1}}
  \pgfmathsetmacro\sinVis{sin(#2)/cos(#2)*sin(\angEl)/cos(\angEl)}
  \pgfmathsetmacro\angVis{asin(min(1,max(\sinVis,-1)))}
  \draw[current plane,thick,purple] (\angVis:1) arc (\angVis:-\angVis-180:1);
  \draw[current plane,thick,dashed, purple] (180-\angVis:1) arc (180-\angVis:\angVis:1);
}
\newcommand\DrawLatitudeCirclethickGREEN[2][1]{
  \LatitudePlane{\angEl}{#2}
  \tikzset{current plane/.prefix style={scale=#1}}
  \pgfmathsetmacro\sinVis{sin(#2)/cos(#2)*sin(\angEl)/cos(\angEl)}
  \pgfmathsetmacro\angVis{asin(min(1,max(\sinVis,-1)))}
  \draw[current plane,thick,green] (\angVis:1) arc (\angVis:-\angVis-180:1);
  \draw[current plane,thick,dashed, green] (180-\angVis:1) arc (180-\angVis:\angVis:1);
}

\tikzset{%
  >=latex,
  inner sep=0pt,%
  outer sep=2pt,%
  mark coordinate/.style={inner sep=0pt,outer sep=0pt,minimum size=3pt,
    fill=black,circle}%
}

\tikzset{%
  >=latex,
  inner sep=0pt,%
  outer sep=2pt,%
  mark purple coordinate/.style={inner sep=0pt,outer sep=0pt,minimum size=3pt,
    fill=purple,circle}%
}

\usepackage{pgfplots}
\usetikzlibrary{intersections, pgfplots.fillbetween}

\begin{document}

\title[Spectral Carleson]{Characterizing higher-dimensional H\"older domains via spherical geometric functions}
\author{Emily Casey, Max Engelstein and Tatiana Toro}
\thanks{E.C. thanks Bobby Wilson for helpful suggestions. All the authors thank Stefan Steinerberger for helpful conversations regarding potential theory and quantitative versions of Faber-Krahn in $\mathbb S^2$. E.C. was partially supported by NSF DMS 2602489,the AMS-Simons Travel Grant, and by the Simons Initiative on the Geometry of Flows  (Grant Award ID BD-Targeted-00017375-ME).  M.E. was partially supported by NSF DMS CAREER 2143719, the George W. Taylor Career Development fund and by the Simons Initiative on the Geometry of Flows  (Grant Award ID BD-Targeted-00017375-ME).  T.T. was partially supported by NSF DMS 2424139 and NSF DMS 1928930.}

    \begin{abstract}
  In this paper we characterize the regularity of a domain in terms of the behavior of a certain spherical function. More precisely, given two domains we show that a certain spherical spectral function (first defined in \cite{FTV23}, see also \cite{AKN21}) has H\"older decay if and only if the domains are complementary and separated by a $C^{1,\beta}$-manifold. This extends the analysis in \cite{Ca24} to dimensions $\geq 3$ and in \cite{FTV23, CTV25} to regularity above the continuous threshold. 

  Along the way, we introduce a spherical analogue of the Jones $\beta$-numbers and show that the decay of these numbers implies geometric regularity of the domain on which they are defined. We believe this result is of independent interest and partially addresses a question of Fleschler-Tolsa-Villa \cite[Question 1]{FTV23}. 
    \end{abstract}

\maketitle

\section{Introduction}\label{s:intro}
In this article we prove a higher regularity version of Carleson's $\varepsilon^2$-conjecture, continuing the study of the following question: given a domain $\Omega\subset \mathbb{R}^n$, is understanding the behavior of $\partial \Omega\cap \partial B(x,r)$, i.e. the behavior of the boundary on spheres, sufficient to determine the regularity of $\partial \Omega$? A question of this nature first arose in the 1980s in \cite{Bi87} following the work of Bishop, Carleson, Garnett, and Jones on a two-phase free boundary problem for harmonic measure, see \cite{BCGJ}. In particular, in \cite{Bi87}, the function below captures the behavior of the boundary of a domain in $\R^2$.

    Let $\Omega^+\subset \mathbb{R}^2$ be such that its boundary, $\partial \Omega^+$, is a Jordan curve, and denote by $\Omega^-:=\mathbb{R}^2\setminus \overline{\Omega}$. Then for $x\in \partial \Omega^+$ and $0<r<\mathrm{diam}(\Omega^+)$, define
    
\begin{equation}\label{e: epsilon function}
    \varepsilon(x,r):=\frac{1}{r}\max\left\{|\pi r-\mathcal{H}^1(I^+(x,r))|,|\pi r-\mathcal{H}^1(I^-(x,r))| \right\},
\end{equation}
where $I^{j}(x,r)$ denotes the longest open arcs of $\partial B(x,r)$ contained in $\Omega^{j}$, for $j=+,-$.

This function is known as the Carleson $\varepsilon$-function (or, the $\varepsilon$-function). 
The $\varepsilon$-function appeared naturally in this context due to its connection with harmonic measure via Beurling's distortion theorem.\footnote{See \cite{Bi87} for more details.} In \cite{Bi87}, the authors exploited this connection to prove that if $x\in \partial \Omega^+$ is a tangent point, then $ \int_0^1\varepsilon(x,r)^2\frac{dr}{r}<\infty$, up to a set of $\mathcal{H}^1$-measure zero. It was conjectured that the converse should also be true. This became known as \textit{Carleson's $\varepsilon^2$-Conjecture}. 

\begin{conjecture}[Carleson's $\varepsilon^2$-conjecture] \label{c: carleson's conjecture} Up to a set of $\mathcal{H}^1$-measure zero, $\partial \Omega$ has a tangent at $x\in \partial \Omega$ if and only if 
    \begin{equation}\label{e: Dini epsilon}
        \int_0^1 \varepsilon(x,r)^2\frac{dr}{r}<\infty.
    \end{equation} 
\end{conjecture}
This conjecture, now a theorem, was open for more than thirty years and was proved in the seminal work of Jaye, Tolsa, and Villa in \cite{JTV21}.\footnote{For a precise definition of tangent point see \cite{BCGJ, GM05, FTV23}.} Even more recently in \cite{FTV23,FTV24}, Fleschler, Tolsa, and Villa proved an analogue of Conjecture \ref{c: carleson's conjecture} in higher dimensions. They proved an almost everywhere characterization of tangent points of higher dimensional domains in terms of a spherical spectral function, which we will denote by $a(x,r)$. We delay the discussion of the function $a(x,r)$ for the moment, but we would like to point out that this spectral function maintains a connection with harmonic measure via the Alt-Caffarelli-Friedman monotonicity formula.

These spherical functions have garnered a lot of interest, and for several reasons. First, these functions are intimately connected to harmonic measure and the growth rate of harmonic functions via the Beurling distortion theorem and the Alt-Caffarelli-Friedman monotonicity formula. Second, the functions {\it a priori} have a lack of cohesiveness across scales, thus requiring delicate analysis (compare this with the $\beta$-numbers, see \eqref{e:jonesbeta}). 

The lack of cohesiveness across scales is a major difficulty in the Carleson conjecture and in related problems, including the one studied in this article (see Theorem \ref{thrm:main theorem}). In order to gain some intuition, let us first consider the case of $n=2$. At scales where the $\varepsilon$-function is small, $\partial \Omega$ cuts circles, in some sense, roughly in half. So at each scale, $\partial \Omega\cap \partial B(x,r)$ is $L^{\infty}$-close to a plane. But, this information is certainly not enough. It is possible for $\varepsilon(x,r)=0$ for all $r\leq 1$, (meaning that at each scale $\partial \Omega\cap \partial B(x,r)$ lies on a line), but 
for the good approximating line at each scale to rotate. See Figure \ref{f: spiral}. 
Thus, despite its simple geometric definition \eqref{e: epsilon function}, even strong control on the $\varepsilon$-function does not give information about the geometric regularity of $\partial \Omega$ apriori. 

\begin{figure}[H]
\centering
\begin{tikzpicture}[scale=1.3]
    \def\angl{180}
		
    
	\draw[domain=0:2350, samples=200, smooth, variable=\x, ultra thick]  plot ({.002*exp(.003*\x)*cos(\x)}, {.002*exp(.003*\x)*sin(\x)}); 
    \draw[domain=0:2350, samples=200, smooth, variable=\x, ultra thick, magenta]  plot ({.002*exp(.003*\x)*cos(\x+\angl)}, {.002*exp(.003*\x)*sin(\x+\angl)}); 
   
	\coordinate (X) at (0,0); 
	\node (Y) at  \polar{2.12}{0} {}; 
    \node (Z) at (1,0) {}; 

 \foreach \x in {0,300,2150}{
 \draw[blue] (0,0) circle ({.002*exp(.003*\x)}); 
 \filldraw[black] ({.002*exp(.003*\x)*cos(\x+\angl)}, {.002*exp(.003*\x)*sin(\x+\angl)}) circle (1pt); 
 \filldraw[black] ({.002*exp(.003*\x)*cos(\x)}, {.002*exp(.003*\x)*sin(\x)}) circle (1pt); 
 }
  \foreach \x in {0,300,1700}{
 \draw[blue] (0,0) circle ({.002*exp(.003*\x)}); 
 \filldraw[black] ({.002*exp(.003*\x)*cos(\x+\angl)}, {.002*exp(.003*\x)*sin(\x+\angl)}) circle (1pt); 
 \filldraw[black] ({.002*exp(.003*\x)*cos(\x)}, {.002*exp(.003*\x)*sin(\x)}) circle (1pt); 
 }
	\end{tikzpicture}
\centering
    \caption{$\varepsilon(0,r)=0$ for all small $r$ but $0$ is not a tangent point of $\partial \Omega$}
\label{f: spiral}
\end{figure}
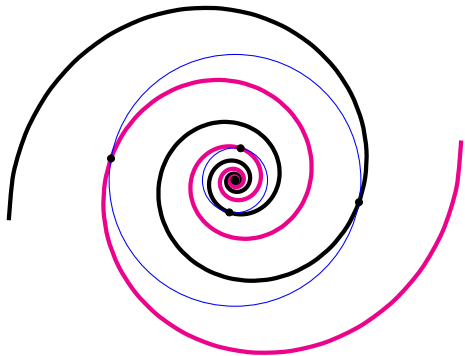

We now return to the discussion of spherical spectral functions. In this article we consider the $\lambda(x,r)$-function, which measures how close the first Dirichlet eigenvalues of both $\Omega^+\cap \partial B(x,r)$ and $\Omega^-\cap \partial B(x,r)$ are to $(n-1)/r^2$, i.e. the first eigenvalue of a hemisphere (see Definition \ref{def: lambda function}), where $\Omega^+$ and $\Omega^-$ are complementary domains.  The $a(x,r)$-function introduced in \cite{FTV23} is slightly different than $\lambda(x,r)$ but when either is small they are comparable (see Remark \ref{r:alphaandlambda}). These spectral functions may not seem like the natural analogue of $\varepsilon(x,r)$ in higher dimensions; however, when $n = 2$ it follows that $\varepsilon(x,r)^2 \approx \lambda(x,r)$ when both are small (see \cite[p.9]{FTV24}). In higher dimensions it follows from \cite[Theorem C]{FTV24} (see also \cite{AKN21}, \cite{BDV}) that the the spectral functions control the natural measure-theoretic generalization of $\varepsilon(x,r)$ (called $\varepsilon_n(x,r)$ see Definition \ref{d:defepsilonn}) and easy examples show the opposite is not true (see Example \ref{ex:lambdavsepsilon} below).

Inspired by \cite{JTV21, FTV23, FTV24}, there has been a surge of activity aimed at understanding the extent to which these spherical functions determine the regularity of the domain on which they are defined. See for instance \cite{Ca24, CTV25, AKN22}, as well as Section \ref{subsec: related work} below. In this article, we contribute to this line of research by proving the first higher regularity version of Carleson's $\varepsilon^2$-Conjecture in higher dimensions. Our main result is the following.

 \begin{theorem}\label{thrm:main theorem}
      Let $n\geq 2$, $\Omega_1,\Omega_2\subseteq \mathbb{R}^n$ be open and disjoint. Let $\Gamma:= \partial \Omega_1\cup \partial \Omega_2$. Fix some $\alpha\in (0,1)$ and $C>0$. If there exists $r_0>0$ such that 
    \begin{equation}\label{eq: lambda infinity}
        \lambda(x,r)< Cr^{\alpha} \qquad \text{for all} \qquad x\in \Gamma, \, r\leq r_0, 
    \end{equation}
   then $\partial \Omega_1 = \partial \Omega_2$ is a $C^{1, \beta}$-manifold, where $\beta$ depends on $\alpha$.
 \end{theorem}

We remark that in the case of $n=2$, Theorem \ref{thrm:main theorem} was proved by the first named author in \cite{Ca24}. Moreover, it is straightforward to show that if $\partial \Omega_1 = \partial \Omega_2$ is $C^{1,\alpha}$ then $\lambda(x,r) < Cr^{\alpha'}$ for any $\alpha' \leq  \alpha$, which means that Theorem \ref{thrm:main theorem} gives a characterization of $C^{1,\alpha}$-domains in $\mathbb{R}^n$. We also point out that no assumptions are made on $\Omega_1, \Omega_2$, or their boundaries. Additionally, the argument is local, in the sense that if the decay in \eqref{eq: lambda infinity} holds for all $x\in B_{2\rho}(x_0)$ then in $B_{\rho}(x_0)$ we have $\partial \Omega_1 = \partial \Omega_2$ is a $C^{1,\beta}$-manifold. 

We further emphasize that the strategy introduced in this article is different than that of the other higher-dimensional papers that consider the same or related spherical functions \cite{FTV23,CTV25}. One difference is that our proof is quantitative, whereas the result in \cite{FTV23} is qualitative and relies on a contradiction compactness argument, and in \cite{CTV25} the authors make a technical assumption which reduces the question of regularity to proving a topological condition. A deeper distinction between the works is how they deal with the issue of ``coherence"; showing that the best approximating hyperplane for $\Gamma \cap \partial B(x,r)$ is not so different from the best approximating hyperplane for $\partial B(y,s)\cap \Gamma$ as long as $s\approx r$ and $|y-x| \leq r$. As stated above, this is a fundamental difficulty in Carleson-type problems. In \cite{FTV23, JTV21} this coherence is achieved via a regularization and compactness argument, which shows that $\Gamma$ is well approximated by a set with smooth structure. In this work, we establish coherence via purely geometric considerations, and prove ``by hand" that the angle between these best approximating planes is controlled by the $\lambda$-function, see e.g. \eqref{e: final angle}. Instrumental to that proof is the introduction of a novel spherical function, $b\beta^s_\Gamma$, which captures the bilateral $L^{\infty}$-closeness of the boundary to a plane \emph{on spheres}. We introduce this function and give a more comprehensive sketch of the argument in Section \ref{ss: strategy}.

 As far as the authors know, this is the first use of a bilateral measure of flatness in spheres (though the $\varepsilon_s$ introduced in \cite{FTV23} can be interpreted as integral analogues). In this article, we show power decay on this spherical function measure of flatness, $b\beta^s$, implies power decay on the classical $\beta$-numbers (see \eqref{e:jonesbeta} and Section \ref{ss: strategy} for more details). Thus, it is natural to ask if one can prove analogous characterizations of regularity with the $b\beta^s$-numbers as with the classical $b\beta_E$-numbers. It is important to note that these spherical measures of flatness give geometric information for a more general class of sets, not just the ``two-sided boundaries" considered in this article. 

\begin{remark}
   Our argument does not give the sharp relationship between $\alpha$ and $\beta$, and so we did not try to keep track of the dependence (though our argument does give an effective bound). We leave this as an open question. However, we have been told that in forthcoming work of Wang and Tolsa, the authors obtain the optimal exponent in the case of $n=2$. 
\end{remark}

\subsection{Related literature}\label{subsec: related work}

Theorem \ref{thrm:main theorem} can be seen as part of a long tradition of theorems concluding regularity from the behavior of localized geometric estimates at many points and scales. Perhaps the first example of this is Reifenberg's topological disc theorem \cite{Re60}. In that work, Reifenberg introduced the following quantity,

\begin{equation}\label{e:jonesbeta}
    b\beta_E(x,r) := \inf_{L\text{ an affine $(d-1)$-space}} \sup_{y\in B(x,r)\cap E} \frac{\mathrm{dist}(y, L)}{r} + \sup_{z\in B(x,r)\cap L} \frac{\mathrm{dist}(z, E)}{r},
\end{equation}
which is a coarse local measurement of how close $E$ is to being a plane (see also \cite{KT97}). It was shown in \cite{Re60} that for each $\alpha \in (0,1)$ there exists an $\varepsilon > 0$ such that if $b\beta_E(x,r) \leq \varepsilon$ for all $x\in E$ and $r < r_0$ then $E$ is a bi-$C^{0,\alpha}$-image of a $(d-1)$-plane. In the intervening 60 years there have been numerous generalizations of this deep insight, see the introduction of \cite{Naber} for a survey of some of these different directions. Moreover, in \cite{J90}, Jones introduced a ``one-sided'' version of \eqref{e:jonesbeta}, $\beta_E$, and showed that the summability of this quantity characterizes subsets of rectifiable curves in the plane. This was later extended to higher dimensions in \cite{O92}. In \cite{BJ94}, Bishop and Jones proved a version of Conjecture \ref{c: carleson's conjecture} in terms of the $\beta_E$, which was later extended to higher dimensions in \cite{HV21}. Again, the literature generalizing Jones' influential idea is too vast to do justice, so we point the reader to the survey article \cite{Badger}. However, we would like to draw particular attention to the work of David, Kenig, and the third named author \cite{DKT}, which we state in a form that is convenient for us, but we note that the actual result is more general:
\begin{theorem}\label{t:DKT}{\cite[Proposition 9.1]{DKT}}
    Suppose $E\subset\mathbb{R}^n$. Fix $\alpha\in (0,1)$ and $C>0$. If there exists an $r_0>0$ such that $b\beta_E(x,r) < Cr^\alpha$ for all $x\in E$ and $r < r_0$, then $E$ is a $C^{1,\alpha}$-manifold. 
\end{theorem}
Roughly, Theorem \ref{t:DKT} says that if locally a set $E$ is quantitatively, bilaterally close to a plane in each ball, then it is locally parameterized by the graph of a $C^{1,\alpha}$-function. Theorem \ref{thrm:main theorem} says instead that if $\partial \Omega$ is close to a plane, in some spectral sense, \emph{in each sphere}, then it is locally parameterized by the graph of a $C^{1,\alpha}$-function. Thus, our Theorem \ref{thrm:main theorem} can be seen as a spherical, spectral, analogue of Theorem \ref{t:DKT}. 

After the resolution of Conjecture \ref{c: carleson's conjecture} in \cite{JTV21}, there has been a lot of activity characterizing domains of a certain degree of regularity in terms of the behavior of spherical functions. As mentioned above, in \cite{FTV23,FTV24}, Fleschler, Tolsa, and Villa proved an analogue of Conjecture \ref{c: carleson's conjecture} in higher dimensions. One of the difficulties in extending Carleson's conjecture to higher dimensions is in defining the ``correct'' higher dimensional analogue of \eqref{e: epsilon function}. For instance, one can imagine two connected complementary domains, $V_1,V_2\subset \mathbb{S}^{n-1}$, with $\mathcal{H}^{n-1}(V_1)=\mathcal{H}^{n-1}(V_2)=\frac{1}{2}\mathcal{H}^{n-1}(\mathbb{S}^{n-1})$, but where $V_1$ and $V_2$ do not agree with complementary hemispheres, i.e. 
$|V_1\Delta H|+|V_2\Delta H^c|$ is large for any hemisphere $H$. 

In \cite{FTV23,FTV24} the authors defined two candidates for a higher dimensional version of the $\varepsilon$-function. The first function is the $\varepsilon_n$-function\footnote{See Definition \ref{d:defepsilonn}.}. The $\varepsilon_n$-function, like its two-dimensional counterpart, uses $\mathcal{H}^{n-1}$-measure to capture how close $\Omega^+\cap \partial B(x,r)$ and $\Omega^-\cap \partial B(x,r)$ are to being complementary hemispheres. Although the $\varepsilon_n$-function may initially seem like a more natural higher-dimensional candidate than the spectral function, $a(x,r)$ discussed above, it is not clear that an analogue of Conjecture \ref{c: carleson's conjecture} should hold if $\varepsilon_n$ satisfies the decay like in \eqref{e: Dini epsilon}. See \cite[Theorem A]{FTV23}, and for a more detailed discussion see \cite[Introduction]{FTV23}.

The second function defined in \cite{FTV23,FTV24} is the spectral $a(x,r)$-function discussed above (see also Definition \ref{d: a function}). 
In \cite{FTV23}, the authors prove a higher dimensional analogue of Conjecture \ref{c: carleson's conjecture} with the $a(x,r)$-function. In \cite{CTV25}, the first author, together with Tolsa and Villa characterized a class of Ahlfors regular domains in terms of equivalent decay on both the $\varepsilon_n(x,r)$ and $a(x,r)$ functions. In other words, if one assumes that $\Gamma$ is Ahlfors regular, then Carleson-measure control on the $\varepsilon_n(x,r)$-function is equivalent to the same Carleson measure control on the $a(x,r)$ function. For related definitions see \cite{CTV25}.  

Thus, a very natural question would be whether the decay in \eqref{eq: lambda infinity} can be replaced with analogous decay on $\varepsilon_n(x,r)$. We leave this as an open question. One observation is that it is not necessarily true that decay of the $\varepsilon_n$-function at a point implies decay of the $\lambda$-function (see Example \ref{ex:lambdavsepsilon}). 

Finally, we mention that Theorem \ref{thrm:main theorem} was proved in the case of $n=2$ with the $\varepsilon$-function by the first named author in \cite{Ca24}. We give a more detailed comparison between the current article and the planar case in Section \ref{ss: strategy}.

\subsection{Strategy of the proof}\label{ss: strategy}
Due to Theorem \ref{t:DKT}, it is sufficient to show that if the $\lambda(x,r)$-function satisfies the decay like in \eqref{eq: lambda infinity}, then $b\beta_{\Gamma}(x,r)$ also satisfies the same H\"{o}lder-decay like in \eqref{eq: lambda infinity}, but with possibly a smaller exponent. In order to get control $b\beta_{\Gamma}$, it is necessary to understand the behavior of $\Gamma$ inside of a ball, but, as we discussed above, the control on the $\lambda$-function in \eqref{eq: lambda infinity} only gives information about $\Omega_1,\Omega_2$, and $\Gamma$ \textit{on spheres}.  
With this in mind, the strategy of the current paper can be summarized in two steps: 
\begin{enumerate}
    \item[(i):](Separation property) show that on each spherical shell, $\partial B(x,r)$, the control on the $\lambda$-function implies that there exist two complementary spherical caps, $A^1:=A^1(x,r)$ and $A^2:=A^2(x,r)$ such that 
    \begin{itemize}
        \item $A^j\subset \Omega^j\cap \partial B(x,r)$, and 
        \item the distance from $A^j$ to a half sphere is smaller than a power of $r$, for $j=1,2$, and
    \end{itemize}
   
    \item[(ii):](Cohesiveness) show that the plane that determines both $A^1$ and $A^2$ (up to a small error) 
    does not oscillate too much as $r\downarrow 0$.
\end{enumerate}

Notice that when $n =2$ the rigid geometry of curves means that (i) follows immediately from the definition of $\varepsilon$. However, in higher dimensions the decay $\lambda(x,r)\leq Cr^{\alpha}$ does not directly imply the the existence of such spherical caps. For example,  consider the following domains in $\mathbb R^3$:
\begin{align*}
    \Omega^1\cap \partial B(x,r)&=\left(\partial B(x,r)\cap \{x_n-x>0\}\right)\setminus \{p^1\}, \qquad \, \text{and}\\
    \Omega^2\cap \partial B(x,r)&=\left(\partial B(x,r)\cap \{x_n-x<0\}\right)\setminus \{p^2\},
\end{align*}
where are $p^1$ and $p^2$ are the north and south poles of $\partial B(x,r)$, respectively. Then 
$$
\lambda_1\left(\Omega^1\cap \partial B(x,r)\right)=\lambda_1\left(\Omega^2\cap \partial B(x,r)\right)=\frac{n-1}{r^2}, 
$$
but there do not exist spherical caps $A^1$ and $A^2$ satisfying (i). Thus to prove (i) in $\partial B(x,r)$ we need not only the smallness of $\lambda(x,r)$ but its smallness at nearby points and comparable scales. However, as a preliminary result we show the smallness of the $\lambda$-function \emph{at a single point} implies the existence of spherical caps $A^1$ and $A^2$ such that
  \begin{itemize}
        \item $A^j$ agrees with $\Omega^j\cap \partial B(x,r)$ \textit{up to sets of quantitatively small relative capacity}, and 
        \item the distance from $A^j$ to a half sphere is smaller than a power of $r$, for $j=1,2$.
    \end{itemize}
In other words, $\partial B(x,r)\cap \partial \Omega$ is well-approximated by a hemisphere, \textit{up to sets of quantitatively small capacity}. The key ingredients in the proof of this weaker capacitary conclusion are the quantitative stability of the first eigenvalue (see Theorem \ref{t:AKN}) and the connection between the change in eigenvalue and the relative Sobolev capacity (see Theorem \ref{thrm: from BC06}). This argument is contained in Lemma \ref{lem: no sets of pos capacity far from BAHP}.

This preliminary estimate seems to be in the spirit of \cite[Theorems B, C]{FTV24}. However, it isn't clear to us whether the results of \cite{FTV24} imply our Lemma \ref{lem: no sets of pos capacity far from BAHP}, due to the technical conditions on the domains of integration in \cite[Theorems B,C]{FTV24}. Regardless, our methods of proof are different, and, in particular, we avoid any consideration of ``thick" parts of the boundary or any mention of Hausdorff content.

We now need to upgrade this weaker capacitary conclusion to the strong geometric conclusion in (i). Crucial to this ``upgrade'' step is the aforementioned introduction of a novel spherical function, the \emph{spherical bilateral $\beta$-numbers}:
\begin{equation}\label{e:sphjonesbeta}
    b\beta^s_E(x,r) := \inf_{L\text{ an affine $(n-1)$-space}} \sup_{y\in \partial B(x,r)\cap E} \frac{\mathrm{dist}(y, L)}{r} + \sup_{z\in \partial B(x,r)\cap L} \frac{\mathrm{dist}(z, E)}{r}.
\end{equation}
The $b\beta^s$-numbers are heavily inspired by the $b\beta$-numbers introduced by \cite{Re60} (see also \cite{J90}) discussed above in Section \ref{subsec: related work}. We emphasize that these $b\beta^s_E$-numbers measure the bilateral closeness of the set $E$ to a plane \textit{on $\partial B(x,r)$}. In our setting, (i) is essentially equivalent to the estimate $b\beta_{\Gamma}^s(x,r)\leq Cr^{\gamma}$ plus a mild ``separation" condition. 

 The difference between the desired control on $b\beta^s_{\Gamma}$, and our weaker condition on the capacity, is that the latter allows for subsets $\Gamma$ with small capacity to be well inside of $A^1$ or $A^2$. As a first step to ruling this out, we prove that the decay of the $\lambda$-function implies that $\Omega_1$ and $\Omega_2$ are complementary corkscrew domains (see Definition \ref{d:corkscrew} and Lemma \ref{lem:existence of corkscrews}).  
 As a consequence, if a single point of $\Gamma$ is well-inside $A^1(x_0,r)$, for instance, we can use the corkscrew condition to construct a sphere centered on $\Gamma$ that has large (in the sense of capacity) intersection with both $A_1(z,\rho)$ and $\Omega_2$, where $|z-x_0| = r$ and $\rho \approx r$. This violates the weaker capacitary condition and thus gives the desired decay for $b\beta_{\Gamma}^s$, see Lemma \ref{lem:no sets of zero capactiy far from BAHP}.

It should be mentioned that in \cite{CTV25}  existence of corkscrews is concluded from weaker control on the $\varepsilon_n$-function but under the additional assumption that $\Gamma$ is Ahlfors regular. Our proof in Lemma \ref{lem:existence of corkscrews} is different and is a straightforward application of the co-area formula and \cite[Lemma 5.3]{FTV23} (this is the only place where we use a result from \cite{FTV23}).

Now that we have shown the desired decay on the $b\beta_{\Gamma}^s$-numbers, it is left to show the cohesiveness of the approximating plane across scales in (ii). Due to the strength of the geometric condition in (i), the controlled oscillation of the planes as $r\downarrow 0$ follows from the triangle inequality and the fact that two $(n-1)$-spheres of similar size with close enough centers intersect in a $(n-2)$-dimensional sphere of comparable radius. We emphasize that the argument only holds in $n\geq 3$, and relies heavily on the rigid geometry obtained in (i). This argument is the content of Lemma \ref{l:sphericaltofull}, and concludes the proof of Theorem \ref{thrm:main theorem}.

We now take a moment to compare this current article with the case of $n=2$ in \cite{Ca24}. At a birds eye view, \cite{Ca24} follows the same strategy as outlined above: first prove a separation property and then prove cohesiveness across scales. As mentioned above, this first step is easier in the plane, see \cite[Lemma 3]{Ca24}. However, due to the disconnected nature of the intersection between two circles in $\mathbb R^2$ the second step, cohesiveness, requires completely different arguments.

\subsection{Outline of the article} We end with a quick outline of the article. In Section \ref{s:prelim} we collect some facts about eigenvalues of spherical domains, discuss the relationships between the $a$, $\lambda$ and $\varepsilon_n$ functions, and define relative capacity and prove a fact about the capacity of rectifiable sets. In Section \ref{s:analytic} we prove the separation property described in (i) above, which includes the proof of the existence of corkscrews, and the argument that the decay in the $\lambda$-function implies the corresponding decay on the $b\beta^s_{\Gamma}$-numbers. Finally in Section \ref{s:geometric} we prove the cohesiveness property described in (ii) above, which is the geometric argument that decay on the $b\beta^s_{\Gamma}$-numbers implies decay on the full $b\beta_{\Gamma}$-numbers. We then conclude the proof of our main Theorem \ref{thrm:main theorem}.

\medskip

\noindent {\bf AI Disclosure:} AI tools were used to convert the the authors' hand drawn diagrams into Tikz code, which is the content of Figures 3-9. Otherwise no AI tools were used in the manuscript.

\section{Examples and Preliminaries}\label{s:prelim}
We start by defining the main geometric functional $\lambda(x,r)$. Let $V$ denote an open subset of $\mathbb{S}^{n-1}$. Recall that the first eigenvalue for the Laplacian with Dirichlet boundary conditions in $V\subset \mathbb{S}^{n-1}$, denoted $\lambda_1(V)$, is defined as follows:
\begin{equation*}
    \lambda_1(V):=\inf_{u\in C_c^\infty(V)}\frac{\int |\triangledown u|^2}{\int|u|^2}. 
\end{equation*}
 Standard arguments show that the infimum in the definition of $\lambda_1(V)$ is achieved in by some $\phi_1^V\in C^\infty(V)$, which satisfies 
\begin{equation}\label{e:eigenfunction}
    \begin{cases}
    -\Delta \phi_1^{V}=\lambda_1(V)\phi_1^{V} &\text{in}\quad V\\
    \phi_1^{V}=0 &\text{on}\quad \partial V.
\end{cases}
\end{equation} 
To be completely precise $\phi_1^V$ will only vanish on the boundary of $V$ up to a polar set. However, this technical point will not come up in our arguments, so we do not elaborate on it further. Finally, if $V$ is connected, $\phi_1^V$ is unique up to a scalar multiple.

A simple scaling argument shows that if $\tilde{V} \subset \partial B(x,r)$ and we define $V = \frac{\tilde{V}-x}{r}\equiv \{y\mid ry + x \in \tilde{V}\} \subset \mathbb S^{d-1}$ then $ \lambda_1(V) = r^2\lambda_1(\tilde{V})$. 

For the remainder of the paper, let $\Omega_1$ and $\Omega_2$ be open and disjoint subsets of $\mathbb{R}^n$, with $n\geq 3$. 
Let $V_i(x,r):=\Omega_i\cap \partial B(x,r)$, which is a relatively open subset of  $\mathbb{S}^{n-1}$. With the above discussion in mind, we define the scale invariant geometric functions which are the main focus of this paper. 

\begin{definition}[$\lambda$-function]\label{def: lambda function} For $x\in \mathbb{R}^n$ and $r>0$. Define
    \begin{equation}\label{e:lambdadefinition}
    \lambda_{\Omega_1,\Omega_2}(x,r):=\left|r^{2}\lambda_1(V_1(x,r))-(n-1) \right|+\left|r^{2}\lambda_1(V_2(x,r))-(n-1) \right|.
\end{equation} 
\end{definition}
Let us make a couple of remarks; first we will often omit the sub-index $\Omega_1, \Omega_2$ if they are obvious in context. Second, if $V$ is an open hemisphere in $\mathbb S^{n-1}$ then $\lambda_1(V) = n-1$. Thus if $V_1(x,r), V_2(x,r)$ are complementary open hemispheres then $\lambda_{\Omega_1,\Omega_2}(x,r) = 0$ (this turns out to be almost an if and only if as we will see below). Thus the $\lambda$ function is a measurement of how far  $\Omega_1, \Omega_2$ are from being complementary half-spheres inside of $\partial B(x,r)$. 

We end this subsection by recalling some basic properties of how the  first Dirichlet eigenvalue of a geodesic ball in the sphere grows with that ball's volume. We introduce a bit of notation: first let $\sigma_n = \mathcal H^{n-1}(\mathbb S^{n-1})$. Then, given $\nu \in (0, \sigma_n)$ let $B_\nu \subset \mathbb S^{n-1}$ be a spherical cap of area $\nu$ (this is unique up to rotation, which we often do not care about but will specify when necessary). It is a straightforward consequence of symmetrization that for any $A\subset \mathbb S^{n-1}$ with volume $\nu$ we have $\lambda_1(B_\nu) \leq \lambda_1(A)$.  

The next result collects some facts about the function $\nu \mapsto \lambda_1(B_\nu)$. We stress that these facts are likely to be well-known to experts but include proofs for the sake of completeness. 

\begin{lemma}\label{l:prelims}
Let $\sigma_n, B_\nu$ be as above. Then
 \begin{itemize}
  \item[(i)] The function $\nu\mapsto \lambda_1(B_\nu)$ is in $C^1((0, \sigma_n))$ and strictly monotone decreasing. 
  \end{itemize}
  Moreover, there exists a dimensional constant, $c_n \in (0,1)$, such that the following holds:
  \begin{itemize}
    \item[(ii)] If $\nu > \frac{3}{4}\sigma_n$ or $\nu < \frac{1}{4}\sigma_n$ then $|\lambda_1(B_\nu) - (n-1)| > c_n$ 
     \item[(iii)] For $\nu_1, \nu_2\in [\frac{1}{4}\sigma_n, \frac{3}{4}\sigma_n]$ we have $$c_n |\nu_1 - \nu_2| \leq |\lambda_1(B_{\nu_1}) - \lambda_1(B_{\nu_2})| \leq c_n^{-1}|\nu_1 - \nu_2|.$$
     
 \end{itemize}
\end{lemma}

\begin{proof}
  Let $B_\nu$ be a spherical cap, which we can rotate so that $B_\nu = \{x\in \mathbb S^{n-1}\mid x_n \geq f(\nu)\}$ where $f: (0, \sigma_n) \rightarrow (-1,1)$ is a strictly monotone decreasing $C^\infty$-smooth function (i.e. the height of the spherical cap increases as the volume decreases). It is then a special case of  \cite[Theorem 5.7.1]{henrot} that $\nu\mapsto \lambda_1(B_\nu)$ is in $C^1((0, \sigma_n))$. Furthermore differentiating this functional (see \cite[equation (A.10)]{ESV} and note, in the notation used there, that we can write $\partial B_{v-t}$ as the graph of a constant function $\xi = f(v) - f(v-t)$ over $\partial B_v$) we have that $\frac{d}{dv}|_{v=v_0}\lambda_1(B_{v}) = c\int_{\partial B_{v_0}}|\partial_{\nu}\phi^{B_{v_0}}_1|^2\, d\mathcal H^{n-1}$. Here the constant $c = -f'(v_0) > 0$ but the precise formula for $c$, as opposed to its sign, is not important for our argument. Note that $-\Delta_{\mathbb S^{n-1}} (-\phi_1^{B_v}) = \lambda_1(B_v)(-\phi_1^{B_v}) \leq 0$. So by the Hopf maximum principle for subsolutions, $|\partial_\nu (- \phi_1^{B_v})| = |\partial_{\nu}\phi^{B_v}_1| >0$ on $\partial B_v$ and so we have proven the first bullet point.

    Since $\lambda_1(B_{\frac{1}{2}\sigma_n})=n-1$, the proof of (ii) follows from (i), taking $\bar{c}_n=\min\{(n-1)-\lambda_1(B_{\frac{3}{4}\sigma_n}),\, \lambda_1(B_{\frac{1}{4}\sigma_n})-(n-1)\}$ (we will pick $c_n$ smaller than $\bar{c}_n$ at the end of the proof).
    To finish the proof we recall from the proof of part (i), that $v_0\mapsto \frac{d}{dv}|_{v=v_0}\lambda_1(B_v)$ is a continuous function which does not vanish for $v_0 \in (0, \sigma_n)$. Thus the map $t\mapsto \lambda_1(B_{\frac{1}{2}\sigma_n + t})$ is bi-Lipschitz for $t\in [-\sigma_n/4, \sigma_n/4]$. Call this bi-Lipschitz constant $\tilde{c}^{-1}_n$. We finish by setting $c_n = \min\{\tilde{c}_n, \bar{c}_n\}$  
\end{proof}

\subsection{Relationship to other spherical functions}\label{subsec: comparing functions}
Given an open set $V\subset \mathbb S^{n-1}$, there exists a unique (non-negative) constant $\alpha = \alpha(V)$ such that $u(r,\theta) \equiv r^\alpha \phi_1^V(\theta)$ is harmonic in the cone over $V$. We call $\alpha(V)$ the \textit{characteristic constant of $V$}. The harmonicity of $u$ in the cone gives the following relationship between $\lambda_1(V)$ and $\alpha(V)$ 
\begin{equation}\label{e: alpha to lambda}
    \alpha(V)=\frac{-(n-2)+\sqrt{(n-2)^2+4\lambda_1(V)}}{2}.
\end{equation}

Note that if $V$ is a hemisphere, then $\lambda_1(V) = n-1$ and $\alpha(V) =1$. Furthermore the Friedland-Hayman inequality \cite{FH76}, states that 
$$
\alpha(V_1(x,r))+\alpha(V_2(x,r))-2\geq 0,
$$
with equality only when $V_1(x,r)$ and $V_2(x,r)$ are complementary hemispheres (up to sets of capacity zero).  This motivates the $a(x,r)$ function, introduced in \cite{FTV24}

\begin{definition}[$a$-function]\label{d: a function} For $x\in \mathbb{R}^n$ and $r>0$. Define
    $$a_{\Omega_1,\Omega_2}(x,r)=\min\{1, \alpha(V_1(x,r))+\alpha(V_2(x,r))-2\},$$ 
    where $\alpha(V)$ is the characteristic constant of $V$, as in \eqref{e: alpha to lambda}. 
\end{definition}

\begin{remark}\label{r:alphaandlambda}
We note that for $V\subset \mathbb{S}^{n-1}$, when $\alpha(V)$ is close to 1, $\lambda_1(V)$ is close to $n-1$, quantitatively. For each $n$ there exists a $C = C(n) > 1$ such that if $|\lambda_1(V) - (n-1)| < 1/8$ then \begin{equation}\label{e:comparativelysmall}
        C^{-1}|\lambda_1(V) -(n-1)| \leq |\alpha(V) - 1| \leq C|\lambda_1(V) - (n-1)|. 
    \end{equation}
Indeed, writing $\lambda_1(V) = (n-1) + x$ we can see that $$\alpha(V) - 1 = \frac{-n + \sqrt{n^2 +4x}}{2} \approx \frac{x}{n}$$ as long as $|x| < 1/8$ (using Taylor's theorem). 

It is also the case that \eqref{e:comparativelysmall} holds when we assume that $|\alpha(V) - 1| < 1/8$. Indeed, 
$$
\lambda_1(V)-(n-1)=\alpha(V)^2+\alpha(V)(n-2)-(n-1)=(\alpha(V)+n-1)\left(\alpha(V)-1\right),
$$
and since $1-\frac{1}{8}\leq \alpha(V)\leq 1+\frac{1}{8}$, \eqref{e:comparativelysmall} follows.

\end{remark}

Indeed, from the remark above it is clear that $\lambda(x,r) \leq Cr^{\alpha}$ implies that $a(x,r) \leq \tilde{C}r^\alpha$ as long as $r < r_0$ for some $r_0$ taken small enough. It turns out that opposite is true as well (up to a loss in the exponent which may not be sharp but is irrelevant to our subsequent discussion):

\begin{lemma}\label{l:alphavslambda}
    Let $\Omega_1, \Omega_2$ be disjoint open subsets of $\mathbb R^n$ and assume for $x_0\in \partial \Omega_1$ and $r < r_0$ we have $a(x_0,r) \leq Cr^\alpha$. There exists a $r_1$ (depending on $r_0$) and a $C_1$ (depending on $n$ and $C$) such that $\lambda(x_0,r) \leq C_1 r^{\alpha/2}$ for all $r < r_1$.
\end{lemma}

\begin{proof} Fix $r\leq r_0$ and let $V_i:=V_i(x_0,r)$.
    The desired result follows from Remark \ref{r:alphaandlambda} as long as we can show that $|\alpha(V_i) - 1| \leq \tilde{C}r^{\alpha/2}$ for each $i= 1,2$ (see \eqref{e:comparativelysmall}. 
    
    Since $a(x_0,r)\leq Cr^{\alpha}$, we have $\alpha(V_1)+\alpha(V_2)\leq 2+Cr^{\alpha}$. Combining this observation with \cite[equation (2.23)]{AKN21}, we have 
    \begin{equation*}
        \frac{|\alpha(V_1)-1|^2}{2+Cr^{\alpha}}+\frac{|\alpha(V_2)-1|^2}{2+Cr^{\alpha}}\lesssim_n \alpha(V_1)+\alpha(V_2)-2\leq Cr^{\alpha}. 
    \end{equation*}
    Thus, for $r\leq 1$,
    \begin{equation*}
        \left|\alpha(V_1)-1\right|\leq C_1r^{\alpha/2},
    \end{equation*}
    where $C_1$ depends only on $C$ and $n$. The result now follows from Remark \ref{r:alphaandlambda}.
   
\end{proof}

The upshot of Lemma \ref{l:alphavslambda} is that we could have written Theorem \ref{thrm:main theorem} in terms of the $a(x,r)$ function. We chose to use the $\lambda$ function because it is more convenient for us to work with eigenvalues than with the associated characteristic constants. However we should note that this passage from the $a(x,r)$-function to the $\lambda(x,r)$-function (and in reverse) relies on our assumption of pointwise smallness and thus it is not clear to us whether other results in the literature, e.g. the Dini-type results of \cite{FTV23}, would work with the $a$-function replaced by the $\lambda$-function.

Our final spherical function, which is more geometric in nature, was defined in \cite{FTV24}. This geometric spherical function also captures how close $\Omega_1 \cap \partial B(x,r)$ and $\Omega_2 \cap \partial B(x,r)$ are to being complementary hemispheres but this time by comparing their measure to that of a hemisphere. See Definition \ref{d:defepsilonn} and Figure \ref{f:epsilonn}
\begin{figure}[h]
    \centering

\begin{tikzpicture}[scale=.65,every node/.style={minimum size=1cm}]


\def\R{4} 

\def\angEl{10} 
\def\angAz{-100} 
\def\angPhiOne{-110} 
\def\angPhiTwo{-45} 
\def\angBeta{30} 


\pgfmathsetmacro\H{\R*cos(\angEl)} 
\LongitudePlane[xzplane]{\angEl}{\angAz}
\LatitudePlane[equator]{\angEl}{0}
\fill[ball color=white!10] (0,0) circle (\R); 

\DrawLatitudeCircle[\R]{0} 


 \def\eps{1.2}
\coordinate (D) at (-4,{.2*\eps});
\def\d{0};

\coordinate (E) at (-2,{-.6*\eps-.5});
\def\e{30};
\coordinate (F) at (-2.4,{.4*\eps-.3});
\def\f{210};
    \coordinate (G0) at (-1.75,{-.15*\eps-.5});
 \def\gz{91};
    \coordinate (G1) at (-1.04,{-.17*\eps-.5});
 \def\ge{310};


\coordinate (G) at (-1,{-.2*\eps-.5});
\def\g{-50};
   
\coordinate (H) at (0,{0*\eps-.4});
\def\h{30};
    \coordinate (H1) at (-.46,{-0.25*\eps-.4});
\def\ha{35};
    \coordinate (H2) at (.44,{-0.24*\eps-.4});
\def\hz{122};

   
\coordinate (I) at (.8,{-.1*\eps-.8});
\def\i{20};
   
\coordinate (J) at (2.6,{.3*\eps-.4});
\def\j{10};
\coordinate (K) at (3.2,{0*\eps-.4});
\def\k{-90};
\coordinate (K1) at (1.35,{-0.2*\eps-.4});
\def\ka{35};

   
    \coordinate (L) at (3,{-.6*\eps-.4});
\def\l{-90};
\coordinate (M) at (3.88,{-.8*\eps});
\def\m{15};
   
 \def\addsm{.9}
\coordinate (T) at (-2,{-.6*\eps+\addsm});
\def\t{40};
   
\coordinate (U) at (-2.4,{.4*\eps+\addsm});
\def\u{-50};
\coordinate (U1) at (-1.93,{-.25*\eps+\addsm});
\def\ua{105};
    \coordinate (U2) at (-1.85,{-.24*\eps+\addsm});
\def\uz{121};


\coordinate (V) at (0,{0*\eps-.5});
\def\v{30};
   
\coordinate (W) at (0,{0*\eps+\addsm});
\def\w{30};
\coordinate (W1) at (-0.16,{-0.18*\eps+\addsm});
\def\wa{80};
    \coordinate (W2) at (0.6,{-0.18*\eps+\addsm});
\def\wz{143};
   
\coordinate (X) at (.8,{-.1*\eps+\addsm-.1});
\def\x{20};
\coordinate (Y) at (2.6,{.3*\eps+\addsm});
\def\y{10};
\coordinate (X1) at (.85,{-.1*\eps+\addsm-.1});
\def\xa{22};
\coordinate (Y1) at (3.23,{-0.4*\eps+\addsm});
\def\ya{95};

   
\coordinate (Z) at (3.2,{0*\eps+\addsm});
\def\z{-90};

    \coordinate (A) at (-2,2);
\def\a{45};
\coordinate (B) at (-1.8,2.2);
\def\b{30};
\coordinate (C) at (-1.6,2.4);
\def\c{10};
\coordinate (N1) at (-1.4,2.3);
\def\n{0};
\coordinate (O1) at (-1,2.2);
\def\o{-100};
\coordinate (P) at (-1.5,2.1);
\def\p{200};
    \coordinate (R) at (-1.9,1.8);
\def\r{190};
    \coordinate (S1) at (-1.2,2.5);
\def\s{-20};

\filldraw[nearly transparent, fill=frontshade] 
    (M) arc (-13.7:175:4) to
(D) to [out=\d,in={\e-180}]
(E) to [out=\e,in={\f-180}]
(F) to [out=\f-180,in={\g-180}]
(G) to [out=\g,in={\h-180}]
(H) to [out=\h,in={\i-180}]
(I) to [out=\i,in={\j-180}]
(J) to [out=\j,in={\k-180}]
(K) to [out=\k,in={\l-180}]
(L) to [out=\l,in={\m-180}]
    (M) -- cycle;

\filldraw[semitransparent, fill=red] 
   (-1.75,-.68) to [out={\gz}, in={\f-180}]
   (-2.4,.18) to  [out={\f-180}, in={\ge-180}]
   (-1.04,-0.69) to [out={178}, in={3}] (-1.75,-.64);

\filldraw[semitransparent, fill=red] 
   (-.46,{-0.25*\eps-.4}) to [out={\ha}, in={\h-180}]
   (0,{0*\eps-.4}) to  [out={\h}, in={\hz}]
   (.44,{-0.24*\eps-.4}) to [out={178}, in={3}] (-.46,{-0.25*\eps-.4});

\filldraw[semitransparent, fill=red] 
   (1.35,{-0.2*\eps-.4}) to [out={\ka}, in={\j-180}]
   (2.6,{.3*\eps-.4}) to  [out={\j}, in={\k-180}]
   (3.2,{0*\eps-.42}) to [out={192}, in={0.5}] (1.35,{-0.2*\eps-.4});

\filldraw[nearly transparent, fill=red] 
   (-1.94,{-.25*\eps+\addsm}) to [out={\ua}, in={\u-180}]
   (-2.4,{.4*\eps+\addsm}) to  [out={\u}, in={\uz-180}]
   (-1.85,{-.24*\eps+\addsm}) to [out={185}, in={0.5}] (-1.94,{-.25*\eps+\addsm});

\filldraw[nearly transparent, fill=red] 
   (-0.16,{-0.18*\eps+\addsm}) to [out={\wa}, in={\w-180}]
   (0,{0*\eps+\addsm}) to  [out={\w}, in={\wz-180}]
   (0.6,{-0.18*\eps+\addsm}) to [out={180}, in={0}] (-0.16,{-0.18*\eps+\addsm});

\filldraw[nearly transparent, fill=red] 
   (.85,{-.1*\eps+\addsm-.1}) to [out={\xa}, in={\y-180}]
   (2.6,{.3*\eps+\addsm}) to  [out={\y}, in={\z-180}]
   (3.2,{0*\eps+\addsm}) to [out={\z}, in={\ya-180}]
   (3.23,{-0.4*\eps+\addsm}) to [out={171}, in={0}] (.85,{-.1*\eps+\addsm-.1});

    \filldraw[nearly transparent, fill=backshade, draw=backshade] 
    (M) arc (-13.7:175:4) to
(D) to [out=\d,in={\t-180}]
(T) to [out=\t,in={\u-180}]
(U) to [out=\u,in={\v-180}]
(V) to [out=\v,in={\w-180}]
(W) to [out=\w,in={\x-180}]
(X) to [out=\x,in={\y-180}]
(Y) to [out=\y,in={\z-180}]
(Z) to [out=\z,in={\m-180}]
    (M) -- cycle;

 \draw[very thick,blue] 
(D) to [out=\d,in={\e-180}]
(E) to [out=\e,in={\f-180}]
(F) to [out=\f-180,in={\g-180}]
(G) to [out=\g,in={\h-180}]
(H) to [out=\h,in={\i-180}]
(I) to [out=\i,in={\j-180}]
(J) to [out=\j,in={\k-180}]
(K) to [out=\k,in={\l-180}]
(L) to [out=\l,in={\m-180}]
    (M);
  \draw[very thick,dashed,blue] 
(D) to [out=\d,in={\t-180}]
(T) to [out=\t,in={\u-180}]
(U) to [out=\u,in={\v-180}]
(V) to [out=\v,in={\w-180}]
(W) to [out=\w,in={\x-180}]
(X) to [out=\x,in={\y-180}]
(Y) to [out=\y,in={\z-180}]
(Z) to [out=\z,in={\m-180}]
    (M);
\filldraw[ semitransparent, fill=red] 
(A) to [out=\a,in={\b-180}]
(B) to [out=\b,in={\c-180}]
(C) to [out=\c,in={\n-180}]
(N1) to [out=\n,in={\s-180}]
    (S1) to [out=\s,in={\o-180}]
(O1) to [out=\o,in={\p-180}]
(P) to [out=\p,in={\r-180}]
(R) to [out=\r,in={\a-180}] (A);

\draw[very thick,blue] 
     (A) to [out=\a,in={\b-180}]
(B) to [out=\b,in={\c-180}]
(C) to [out=\c,in={\n-180}]
(N1) to [out=\n,in={\s-180}]
    (S1) to [out=\s,in={\o-180}]
(O1) to [out=\o,in={\p-180}]
(P) to [out=\p,in={\r-180}]
(R) to [out=\r,in={\a-180}] (A);

\coordinate[mark coordinate] (O) at (0,0);
\coordinate (N) at (0,{\H+.05});
\coordinate (S) at (0,{-\H-.05});


\node at (0.2,.2) {$x$};

\filldraw[ semitransparent, fill=red] 
(-4,0) arc (180:219: 4 cm and .7 cm) to[out=125, in=\d] (D) --cycle;
\end{tikzpicture}
    \caption[The region $(\partial B(x,r)\cap H^+)\setminus \Omega_1$ is denoted in red.]{The region $(\partial B(x,r)\cap H^+)\setminus \Omega_1$ is denoted in red.
    }
     \label{f:epsilonn}
\end{figure}
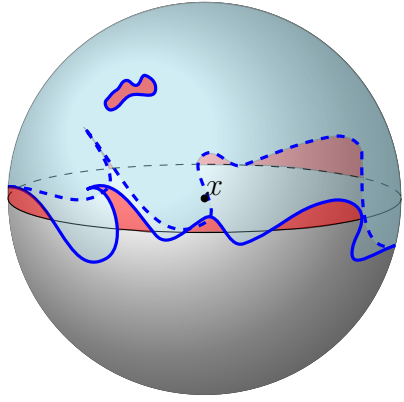
\begin{definition}\label{d:defepsilonn}($\varepsilon_n$-function)
    For $x\in \mathbb{R}^n$ and $r>0$, define 

    \begin{equation}\label{e:defepsilonn}
\varepsilon_n(x,r):=\frac{1}{r^{n-1}} \inf_{H^+}\mathcal{H}^{n-1}\bigg(\big[(\partial B(x,r)\cap H^+)\setminus \Omega_1\big] \cup \big[(\partial B(x,r)\cap H^-)\setminus \Omega_2\ \big]\bigg), 
\end{equation}
where the infimum is taken over all half-spaces $H^+$ with $x\in \partial H^+$.\footnote{Figure \ref{f:epsilonn} was created using code based on \cite{trzeciak2008stereographic,miani2009spherical}}
\end{definition}

To understand how this function relates to the $\lambda$-function, we recall the main result of \cite{AKN21} (see also \cite{BDV} for an earlier but weaker result):

\begin{theorem}[\cite{AKN21}, Theorem 1.1]\label{t:AKN}
    Fix $0<\nu<\sigma_n$. There exists a constant $c:=c(n, \nu)$ such that for any open set $V\subset \mathbb{S}^{n-1}$ with $\mathcal H^{n-1}(V)=\nu$ we have 
    \begin{equation}\label{eq: AKN}
        \lambda_1(V)-\lambda_1(B_\nu)\geq c \inf_{x\in \mathbb{S}^{n-1}}\left( \mathcal{H}^{n-1}(V\Delta B(x))^2+\int_{\mathbb{S}^{n-1}}|\phi_1^{V}-\phi_1^{B(x)}|^2dx\right),
    \end{equation}
    where $B(x)$ is a spherical cap centered at $x$ with radius uniquely chosen so that $\mathcal H^{n-1}(B(x))=\nu$. Recall that $B_\nu$ denotes any spherical cap with volume $\nu$.
\end{theorem}

We should clarify that when we refer to the \emph{center} of a spherical cap, we are thinking of that cap as a geodesic ball inside of the sphere and are referring to the center of the geodesic ball. If $x_0$ realizes the infimum above, we write $B_V:=B(x_0)$ and call it the \textit{best approximating spherical cap} 
for $V$. 

An easy corollary of Theorem \ref{t:AKN} and the standard facts about $\lambda_1$ (i.e. Lemma \ref{l:prelims}) is the following comparison between the $\lambda(x,r)$ and $\varepsilon_n(x,r)$. 

\begin{proposition}\label{p:FTV thrm c} Let $\Omega_1$ and $\Omega_2$ be open and disjoint subsets of $\mathbb{R}^n$ and define $\varepsilon_n, \lambda$ as above with respect to these domains. There exists a dimensional constant $C > 0$ such that for any $x\in \mathbb{R}^n$ and $r>0$,
\begin{equation*}
    \varepsilon_n(x,r)^2\leq  C\lambda(x,r).
\end{equation*}
\end{proposition}

This result should be compared with \cite[Theorem C]{FTV24}, which proves a similar comparison with the $a(x,r)$-function and a one-parameter family of $\varepsilon_n$-type functions. In fact, the proof of Proposition \ref{p:FTV thrm c} also follows directly from \eqref{e:comparativelysmall} and  \cite[Theorem C]{FTV24}, but we provide an independent proof for completeness. 

\begin{proof}
    First we can assume that there is some $\delta_0 > 0$ depending only on dimension such that $\lambda(x,r) \leq \delta_0$ (otherwise the inequality is trivial since $\varepsilon_n(x,r) \leq \sigma_n$). 
    
    Since everything is scale invariant we can let $x= 0 $ and $r = 1$. Adopting the notation from above let $V_i = \Omega_i \cap \mathbb S^{n-1}$ and $\nu_i = \mathcal H^{n-1}(V_i)$. Let $B_i$ be the best approximating spherical cap for $V_i$ and let $H_i$ be a hemisphere
    with the same center as $B_i$.  Taking $\delta_0$ small enough, from Lemma \ref{l:prelims} (ii) it follows that $\frac{1}{4}\sigma_n\leq \nu_i\leq \frac{3}{4}\sigma_n$, for $i=1,2$.

    Let us assume, without loss of generality, that $\frac{1}{4}\sigma_n\leq \nu_1 \leq \frac{1}{2}\sigma_n$, which implies that $\lambda_1(V_1) \geq \lambda_1(B_1) \geq n-1$. We then estimate, using Theorem \ref{t:AKN} and Lemma \ref{l:prelims}, 
    \begin{equation}\label{e:lambdacontrolsall}\begin{aligned}
    \lambda_1(V_1) - (n-1) &\geq \lambda_1(V_1) - \lambda_1(B_1) + \lambda_1(B_1) - (n-1)\\
    &\geq c\left(\mathcal{H}^{n-1}(V_1\Delta B_1)^2 + |\nu_1-\frac{1}{2}\sigma_n|\right) \\ &\geq \tilde{c}\left(\mathcal{H}^{n-1}(V_1\Delta B_1)^2 + \mathcal{H}^{n-1}(B_1\Delta H_1)^2\right).
    \end{aligned}
    \end{equation} 
   Here we have used two facts: $\sigma_n\mathcal{H}^{n-1}(B\Delta H) \geq \mathcal{H}^{n-1}(B\Delta H)^2$ for any spherical cap $B$ and hemisphere $H$, and $|\nu_1-\frac{1}{2}\sigma_n| = \mathcal{H}^{n-1}(B_1\Delta H_1)$. From \eqref{e:lambdacontrolsall} and the triangle inequality we have, \begin{equation}\label{e:lambdaclosetohemi}\lambda_1(V_1) - (n-1) \geq c\mathcal{H}^{n-1}(V_1\Delta H_1)^2.\end{equation}

    Moreover, since $V_1\subset V_2^c$ and $V_2\subset V_1^c$,
    $$
    \mathcal{H}^{n-1}(V_2 \Delta H_1^c) \leq \mathcal{H}^{n-1}(V_1\Delta H_1) + \mathcal{H}^{n-1}(H_1^c\setminus (V_1\cup V_2))\leq \mathcal{H}^{n-1}(V_1\Delta H_1)+\sigma_n - \nu_1 - \nu_2.
    $$
    
    Since we have,
    \begin{equation}\label{e:intermediate}
        \varepsilon_n(0,1)^2 \lesssim \mathcal{H}^{n-1}(V_1\Delta H_1)^2+\mathcal{H}^{n-1}(V_2\Delta H_1^c)^2\lesssim \lambda_1(V_1)-(n-1)+(\sigma_n-\nu_1-\nu_2)^2,
    \end{equation}
    we are done once we show that $\lambda(0,1) \gtrsim (\sigma_n-\nu_1-\nu_2)^2$. To this end, let us consider two cases. If $\nu_2\leq \frac{1}{2}\sigma_n$, then as in \eqref{e:lambdacontrolsall} we have 
    $$
    \lambda_1(V_2)-(n-1)\geq \frac{1}{2}\sigma_n-\nu_2,
    $$
    and thus 
    $$\lambda(0,1) = \lambda_1(V_1) - (n-1) + \lambda_1(V_2) - (n-1) \geq \sigma_n - (\nu_1 + \nu_2).
    $$

    If $\nu_2\geq \frac{1}{2}\sigma_n$ then,
    \begin{align*}
        \nu_1+\nu_2&\geq \frac{1}{2}\sigma_n-\lambda_1(V_1)+(n-1)+\frac{1}{2}\sigma_n\geq \sigma_n-\lambda(0,1).
    \end{align*}
    Thus, in either case
    $\sigma_n - \nu_1 - \nu_2\leq \lambda(0,1)$.

    Thus, applying this estimate in \eqref{e:intermediate} finishes the proof.

\end{proof}

We end this subsection with two examples promised in the introduction:

\begin{example}\label{ex:lambdavsepsilon}[$\lambda$ is more sensitive than $\varepsilon_n$]
    As mentioned above, when $n =2$ we see that $\lambda(x,r)\approx \varepsilon(x,r)$ (as long as both are small).  Furthermore, since $\varepsilon_n^2(x,r) \lesssim \lambda(x,r)$ (see Proposition \ref{p:FTV thrm c}, and also \cite[Theorem C]{FTV24}). The purpose of this example is to give a domain $\Omega \subset \mathbb R^3$ such that $0\in \partial \Omega$ with $\varepsilon_3(0, r) \leq Cr^\alpha$ but $\lambda(0, r)$ is not Dini-summable. 

    Indeed let $$\Omega_1 = \{(x,y,z)\in \mathbb R^3\mid z > 0\}\setminus\{(x,0,z)\mid z \geq |x| \}.$$ and let $\Omega_2$ be the reflection of $\Omega_1$ across the plane $\{z = 0\}$.  Note that both $\Omega_i$ are connected.

 We see that $\partial B(0,r)\cap \Omega_1$ is the upper-half hemisphere minus the curve $$\gamma_r := \{(t, 0, \sqrt{r^2-t^2})\mid t\in [-r/\sqrt{2}, r/\sqrt{2}]\}.$$ Since $\mathcal H^2(\gamma_r) = 0$ we have that $\varepsilon_3(0, r) \equiv 0$ for all $r > 0$. On the other hand,  using Lemma \ref{lem: relative cap to cap} and Theorem \ref{thrm: from BC06} we can see that $\lambda_1(\partial B(0,r)\cap \Omega_1) \geq n-1 + c$ for some constant $c > 0$. Thus $\lambda(0,r) \geq 1$ for all $r > 0$.

 It is important to note that this example does not rule out a version of Theorem \ref{thrm:main theorem} with H\"older decay on the $\varepsilon_n$ (as opposed to the $\lambda(x,r)$). Indeed, even though $\Omega_1$ is not a $C^{1,\alpha}$-domain it is also the case that $\varepsilon_3((0,0,z),r) \gtrsim 1$ for any $r < |z|$.
\end{example}

Our last example shows that decay of $\lambda$ at a single point cannot imply decay of the spherical $\beta$s at that point:

\begin{example}\label{ex:lambdavsbeta}
Let $\Omega_1 \subset \mathbb R^3$ with $\Omega_1 = \{(x,y,z)\mid z > 0\} \setminus 
\{(0,0,z):z\in \mathbb{R}\}$. Let $\Omega_2$ be the reflection of $\Omega_1$ over the plane $\{z=0\}$. Then $\lambda(0, r) \equiv 0$ because $\Omega_i \cap \partial B(0,r)$ is a hemisphere minus the north (or south) pole (which has zero capacity). On the other hand $b\beta^s_{\partial \Omega_i}(0, r) \simeq \sqrt{2}$ because the pole is far from the rest of the boundary which is at the equator. 
\end{example}

\subsection{Capacity}\label{ss:capacity}
A key tool in our arguments will be the notion of relative capacity. For $\Omega\subset\mathbb{S}^{n-1}$, and $A\Subset \Omega$ we define \textbf{the capacity of $A$ relative to $\Omega$} as
\begin{equation*}
    \mathrm{Cap}_{\Omega}(A)=\inf \left\{\int_{\mathbb{S}^{n-1}}|\nabla u|^2d\mathcal{H}^{n-1}(x)\mid \,
    \,u- \phi^\Omega_1 \in C_c^\infty(\Omega \backslash \overline{A})\right\}, 
\end{equation*}
where $\phi_1^{\Omega}$ is the first eigenfunction of the Laplace-Beltrami operator on $\Omega\subset \mathbb{S}^{n-1}$. See \eqref{e:eigenfunction}. 
\vspace{\baselineskip}

This definition of relative capacity can be found in \cite{BC06}. We remark that there is a well known relationship between Sobolev type capacities (like the ones above) and the integral capacities used by \cite{FTV23} (see \cite{Ma99}). For our purposes it is easier to work with Sobolev capacities.

One key property of capacity is that ``nice enough" $(n-2)$-dimensional sets (in $\mathbb S^{n-1}$) have positive relative capacity. In particular, when $\Omega$ is smooth and $A$ is a subset of a $(n-2)$-rectifiable set  with $\mathrm{dist}_H(A, \Omega) \geq \varepsilon$, then $\mathrm{Cap}_{\Omega}(A)$ is proportional to some (power of) $\varepsilon$ times the measure of $A$. Here $\mathrm{dist}_H(U,V)$ denotes the Hausdorff distance between $U$ and $V$.  
We prove this in the special case where $\Omega = H$, a hemisphere.

\begin{lemma}\label{lem: relative cap to cap}
    Let $H\subset \mathbb{S}^{n-1}$ denote a hemisphere. Let $A$ be a subset of $\Sigma$, where $\Sigma$ is a smooth $(n-2)$-hypersurface such that $d(\Sigma,\partial H)>\varepsilon$. Now assume that there exist $\rho \ll \varepsilon$ (where how much less only depends on $n$) and point $x_0\in A$ such that $\mathcal H^{n-2}(A\cap B(x_0, \rho)) \geq \theta \rho^{n-2}$ for some $\theta \in (0,1)$. Then there exists a constant $C = C(n, \theta) > 0$ such that 
    \begin{equation}\label{e: relative cap to cap}
       \varepsilon^2 \rho^{n-1}\leq C\mathrm{Cap}_{H}(A).
    \end{equation}

\end{lemma}
\begin{proof}
Suppose for convenience that $H=\{x_{n}>0\}\cap \mathbb{S}^{n-1}$. Let $\phi_1:=\phi_1^H$ denote the first eigenfunction of the Laplace-Beltrami operator on $H$ with Dirichlet boundary conditions, and observe that $\phi_1 = c_n x_n^+$. Let $u$ be any function such that $u-\phi_1 \in C_c^\infty(H\backslash \overline{A})$ and consider $v := \max\{\phi_1 - u, 0\}$. Finally let $B = B(x_0, \rho)$ as in the Lemma statement. We have two cases: 

\noindent {\bf Case 1:} We assume that $\int_{B} v^2 \leq \frac{1}{4}\int_B \phi_1^2$. From this assumption and the reverse triangle inequality,
$$
\int_{B\cap \{\phi_1 > u\}} u^2 \geq \frac{1}{2}\int_{B\cap \{\phi_1 > u\}} \phi_1^2 - \frac{1}{4}\int_{B}\phi_1^2.
$$ 
We now apply the estimate above to estimate $\int_B u^2$, 
$$
\int_B u^2\geq \int_{B\cap \{\phi_1 > u\}} u^2 +\int_{B\setminus \{\phi_1 > u\}} \phi_1^2 \geq 
\frac{1}{4}\int_B \phi_1^2,
$$
and using our precise description of $\phi_1 = c_n x_n^+$ we have that
$$
\int_B u^2\geq C\varepsilon^2\rho^{n-1},
$$ 
where $C$ depends only on the dimension. Note that $u$ is a competitor for $\lambda_1(H)$ and so $$\int_{\mathbb{S}^{n-1}}|\nabla u|^2 \geq (n-1)\int_{\mathbb{S}^{n-1}} u^2.$$ Putting everything together and taking an infimum over $u$ we have the desired lower bound in this case. 

\medskip

\noindent {\bf Case 2:} We now assume the opposite, that $\int_B v^2 \geq \frac{1}{4}\int_B \phi_1^2$. Since $\Sigma$ is a smooth (rectifiable would have been sufficient) and $v$ vanishes on $A \subset \Sigma$ we can apply the Poincar\'e inequality (see, e.g. \cite[Theorem 3.11.1]{Ziemer}) and get that there exists a constant $C > 0$ which depends on $\theta,n$ such that  
$$
C\rho^2\int_B |\nabla v|^2 \geq \int_B v^2.
$$
Using the triangle inequality and the lower bound for $\int_B v^2$ we have that 
$$C\rho^2 \int_{B\cap \{\phi_1>u\}} |\nabla u|^2 + C\rho^2 \int_{B\cap \{\phi_1>u\}} |\nabla \phi_1|^2 \geq C\rho^2 \int_B |\nabla v|^2 \geq \int_B v^2 \geq \frac{1}{4}\int_B \phi_1^2\approx c\rho^{n-1}\varepsilon^2.
$$ 
However we know that $|\nabla \phi_1| \approx c_n$, so, 
$$  C\rho^2 \int_{B\cap \{\phi_1>u\}} |\nabla u|^2 +  C\rho^{n+1} \geq c\rho^{n-1}\varepsilon^2 \quad \text{which implies}\quad  C\int_B |\nabla u|^2 \geq \rho^{n-3}\varepsilon^2.$$ 
This bound is much larger than $\rho^{n-1}\varepsilon^2$, since $\rho \ll 1$, so we are done (after taking an infimum over $u$). 
 
\end{proof}

We end this section with the following useful theorem, which allows us to bound relative capacity by a drop in principle eigenvalues. This theorem is due to to Bertrand and Colbois, \cite{BC06}, in $\mathbb R^n$. The proof for smooth subsets of the sphere is identical, but we give it here to make the paper self contained.

\begin{theorem}[Theorem 2.2, \cite{BC06}]\label{thrm: from BC06} Let $\Omega\subset \mathbb{S}^{n-1}$ be an open subset with smooth boundary. Let $\overline{A}\subset \Omega$. Then 
\begin{equation*}
    \mathrm{Cap}_\Omega(A)\leq \kappa_\Omega\left(\lambda_1(\Omega\setminus \overline{A})-\lambda_1(\Omega) \right),
\end{equation*}
    where $\kappa_\Omega$ can be computed as follows 
    \begin{equation*}
        \kappa_\Omega=\frac{\lambda_1(\Omega)+\lambda_2(\Omega)}{\lambda_2(\Omega)-\lambda_1(\Omega)}.
    \end{equation*}
    Recall that $\lambda_i(S)$ is the $i^{th}$ Dirichlet eigenvalue of the Laplace-Beltrami operator on $S\subset \mathbb{S}^{n-1}$. 
\end{theorem}

\begin{proof}
    Let $u \in C_c^\infty(\Omega \backslash \overline{A})$ with $\|u\|_{L^2} = 1$. Since $\Omega$ has smooth boundary, classical spectral theory applies. In particular, the Dirichlet eigenfunctions are an orthonormal basis of $H^1_0(\Omega)$. Since $u\in H^1_0(\Omega)$, we can write $u = c_u\phi_1 + \psi$, where $\phi_1$ is the first Dirichlet eigenfunction of $\Omega$ and $\psi \in H^1_0(\Omega)$, with $\int_{\mathbb{S}^{n-1}}\psi \phi_1=0$. Note this is a consequence of projection in the Hilbert space $H^1_0(\Omega)$, and in particular we are making no claims about the properties of $\psi$, besides its orthogonality to $\phi_1$.  

    Finally, $\phi_1\pm u$ are both competitors for the relative capacity, so we can write 
    $$
    \mathrm{Cap}_\Omega(A) \leq \int |\nabla (\phi_1\pm u)|^2 = \int |\nabla u|^2 + \lambda_1(\Omega) \pm 2\int \nabla u \nabla \phi_1.
    $$
    Moreover, integrating by parts yields
    $$
    2\int \nabla u \nabla \phi_1 = 2\lambda_1(\Omega)\int u \phi_1,
    $$ 
    where we have used that $u\in C^\infty_c(\Omega \backslash \overline{A})$. 

    Using the decomposition of $u$, we have that $\int u \phi_1 = c_u$. Thus we have shown that $$\mathrm{Cap}_\Omega(A) \leq \int |\nabla u|^2+\lambda_1(\Omega) \pm 2c_u\lambda_1(\Omega) \leq (\int |\nabla u|^2 - \lambda_1(\Omega)) + 2(1\pm c_u)\lambda_1(\Omega).$$

 Since $\int_{\mathbb{S}^{n-1}}u^2=1$ it follows that $c_A^2 + \int \psi^2 = 1$, and that $$ \int |\nabla u|^2 \geq  c_A^2 \lambda_1(\Omega) + (1-c^2_A)\lambda_2(\Omega)$$  (since if $\psi$ is orthogonal to $\phi_1$ we have $\int |\nabla \psi|^2 \geq \lambda_2(\Omega)\int \psi^2$). Thus $$\frac{\int |\nabla u|^2-\lambda_1(\Omega)}{\lambda_2(\Omega) - \lambda_1(\Omega)} \geq (1-c_u^2) \geq (1-|c_u|).$$ 

 We now choose the sign above so that $\pm c_A = -|c_A|$. Plugging that into the above we get that $$\mathrm{Cap}_{\Omega}(A) \leq \left(\frac{2\lambda_1(\Omega)}{\lambda_2(\Omega) - \lambda_1(\Omega)} + 1\right)\left(\int |\nabla u|^2 - \lambda_1(\Omega)\right).$$ Taking the infimum over all such $u$ we get the desired result.  
\end{proof}

\section{The Analytic Argument}\label{s:analytic}
The goal of this section is to show that if $\lambda(x,r)$ decays at a power rate, then so do the spherical beta numbers, $b\beta^s(x,r)$. The first step is to show that on a spherical shell where the $\lambda$-function is small, there exists a hemisphere which well approximates $\Omega$ in the sense of capacity. Recall, capacity is more sensitive than measure, but still does not imply $L^\infty$ closeness. We reiterate that this estimate feels similar in spirit to \cite[Theorems A, B]{FTV24}, which shows that the eigenvalue deficit controls integral quantities related to capacity and Hausdorff content. 

The first lemma we prove to this end, says that if $H$ is a good approximating hemisphere for a set $\Omega\subset \mathbb{S}^{n-1}$, then $\overline{\Omega^c}\cap H$ has small capacity away from a neighborhood of $\partial H$.

 \begin{lemma}\label{lem: no sets of pos capacity far from BAHP}
There exists $0<\varepsilon_0\ll 1$ depending on dimension, so that for any $\varepsilon<\varepsilon_0$ the following holds. If $\Omega \subset \mathbb S^{n-1}$ is an open set with $|\lambda_1(\Omega)-(n-1)|<\varepsilon$ and $|\mathcal{H}^{n-1}(\Omega) - \frac{1}{2}\sigma_n|^2 < \varepsilon$, 
then there exists a constant $C:=C(n) > 0$ and a hemisphere $H$ such that 
\begin{equation}\label{e:nolargecap}
\mathrm{Cap}_H( \overline{\Omega^c}\cap  H_{\varepsilon^{\beta}})< C\varepsilon^{1/3}, 
\end{equation}
for any $\beta<1/12$. 
 Here we use the notation $H_r := \{x\in H\mid \mathrm{dist}(x, \partial H) \geq r\}$.

\end{lemma}

\begin{proof}
Let $H$ be a hemisphere such that $|\Omega \Delta H|^2 \leq C\varepsilon$. We note that such a hemisphere exists by arguing as in the proof of Proposition \ref{p:FTV thrm c}, (in particular, see  \eqref{e:lambdaclosetohemi}). Let $A=\overline{\Omega^c} \cap H_{\varepsilon^{\beta}}$ and note that $A=\overline{A}$. Since $A\subset H$, it follows from Theorem \ref{thrm: from BC06} that 
\begin{equation}\label{e:001}
    \mathrm{Cap}_H(A)\leq \kappa_H\left(\lambda_1(H\setminus \overline{A})-(n-1) \right),
\end{equation}
where $\kappa_H=\frac{\lambda_1(H)+\lambda_2(H)}{\lambda_2(H)-\lambda_1(H)}$, which is a dimensional constant since $H$ is a hemisphere.\footnote{We can compute $\kappa_H=\frac{3n-1}{n+1}$, as $\lambda_1(H)=n-1$ and $\lambda_2(H)=2n$.} 
So, we are done if we can prove $\lambda_1(H\setminus A)\leq \lambda_1(\Omega)+C\varepsilon^{1/3}$, since $\lambda_1(\Omega) \leq n-1 + \varepsilon$ by assumption.

Define 
\begin{equation*}
    \eta=\begin{cases}
        1&\text{on}\qquad H_{2\varepsilon^{1/3}}\\
        0&\text{on}\qquad \mathbb{S}^{n-1}\setminus H_{\varepsilon^{1/3}}, 
    \end{cases}
\end{equation*}
such that $|\nabla \eta|\lesssim \varepsilon^{-1/3}, |\Delta \eta| \lesssim \varepsilon^{-2/3}$. For ease of notation let $\phi_1:=\phi_1^\Omega$. 
If $\psi_i \in C_c^\infty(\Omega)$ are such that $\psi_i \rightarrow \phi_1$ in $H^1(\Omega)$, then  $\eta \psi_i \in C_c^\infty(H\setminus A)$ and $\eta\psi_i \rightarrow \eta \phi_1$ in $H^1(H\setminus A)$ so 
\begin{equation}\label{e: comp bound}
    \lambda_1(H\setminus A)\leq \frac{\int_{\mathbb{S}^{n-1}} |\nabla(\eta \phi_1)|^2dx}{\sint \eta^2\phi_1^2dx}.
\end{equation}
We first bound $\int_{\mathbb{S}^{n-1}} |\nabla(\eta \phi_1)|^2dx$, writing 
\begin{equation}\label{e: comp expanded}
    \int_{\mathbb{S}^{n-1}} |\nabla(\eta \phi_1)|^2dx=\sint \eta^2|\nabla \phi_1|^2dx+2\sint \eta \phi_1 \nabla \phi_1\cdot \nabla \eta dx+\sint \phi_1^2|\nabla \eta|^2dx. 
\end{equation}
Let $I= \sint \eta \phi_1 \nabla \phi_1\cdot \nabla \eta$. Integrating by parts yields, 
\begin{equation}\label{e: 2I}
    2I=-\sint \eta \phi_1^2 \Delta \eta dx -\sint \phi_1^2|\nabla \eta|^2dx.
\end{equation}
 Substituting \eqref{e: 2I} into \eqref{e: comp expanded} gives 
\begin{equation}\label{e: comp simplified 1}
    \int_{\mathbb{S}^{n-1}} |\nabla(\eta \phi_1)|^2dx=-\sint \eta \phi_1^2 \Delta \eta dx +\sint \eta^2|\nabla \phi_1|^2dx. 
\end{equation}
Recall that $B:=B_{\Omega}$ is the spherical cap on $\mathbb{S}^{n-1}$ such that $|B_{\Omega}|=|\Omega|$ and $B$ is chosen to minimize the right hand side of \eqref{eq: AKN}. Let $\phi_B = \phi_1^{B_{\Omega}}$ denote the first eigenfunction of $B_\Omega$. Then, 
\begin{equation*}
     \int_{\mathbb{S}^{n-1}} |\nabla(\eta \phi)|^2dx\leq \lambda_1(\Omega)+2\sint |\phi_1-\phi_B|^2|\eta \Delta \eta|dx + 2\sint \phi_B^2|\eta \Delta \eta|dx.  
\end{equation*}
To bound the middle term on the right-hand side above, let $\nu_{\Omega}:=\mathcal{H}^{n-1}(\Omega)$, and we observe that from the hypotheses on $\Omega$, Theorem \ref{t:AKN}, and Lemma \ref{l:prelims} that we can bound
\begin{equation}\label{e:eigenfunction bound}
    \sint |\phi_1-\phi_B|^2\leq| \lambda_1(\Omega)-(n-1)|+|\frac{1}{2}\sigma_n-\nu_{\Omega}|<2\varepsilon.
\end{equation}

Combining this observation with the bound on $|\Delta \eta|$ and the fact that $\mathrm{spt} |\Delta \eta| \subseteq H_{\varepsilon^{1/3}}\setminus H_{2\varepsilon^{1/3}}$ we obtain 
\begin{equation}\label{e: num almost}
     \int_{\mathbb{S}^{n-1}} |\nabla(\eta \phi_1)|^2dx\leq \lambda_1(\Omega) +2c\varepsilon^{1/3} + 2c\varepsilon^{-2/3}\int_{H_{\varepsilon^{1/3}\setminus}H_{2\varepsilon^{1/3}}} \phi_B^2dx. 
\end{equation}
To handle the last term in the above equation, we again notice from the hypotheses on $\Omega$ and Theorem \ref{t:AKN} that 
$$
|B\Delta H| \leq |H\Delta \Omega| + |\Omega \Delta B| \leq C\varepsilon+C\varepsilon^{1/2}\leq C\varepsilon^{1/2}.
$$
Since $B$ and $H$ are spherical caps, we deduce that $\partial B$ is in a $C\varepsilon^{1/2}$ neighborhood of $\partial H$. 

Observe that $\mathrm{spt}\, \phi_B \subseteq B$, and $\phi_B\lesssim d(x,\partial B)$ (depending on a dimensional constant, as long as $\varepsilon$ is small enough). Thus, for $\varepsilon$ small enough (depending on the ambient dimension) 
\begin{equation}\label{e: phiB bound}
    \phi_B(x)\lesssim 2\varepsilon^{1/3}+\varepsilon^{1/2}=C \varepsilon^{1/3},\;\; \text{for all}\;\; x\in H_{\varepsilon^{1/3}}\backslash H_{2\varepsilon^{1/3}}. 
\end{equation}
Applying this to the estimate in \eqref{e: num almost} gives 
\begin{equation*}
     \int_{\mathbb{S}^{n-1}} |\nabla(\eta \phi_1)|^2\leq\lambda_1(\Omega)  
     + C\varepsilon^{1/3}.  
\end{equation*}
Thus the bound in \eqref{e: comp bound} becomes, 
\begin{equation}\label{e: num bound}
    \lambda_1(H\backslash A) \leq \frac{\lambda_1(\Omega) + C\varepsilon^{1/3}}{\int_{\mathbb S^{n-1}}\eta^2 \phi_1^2}.
\end{equation}

It is left to estimate $\sint \eta^2\phi_1^2$. Indeed, from Theorem \ref{t:AKN} (see also \eqref{e:eigenfunction bound}), 
\begin{equation}\label{e: denom bound 1}
    \sint \eta^2\phi_1^2dx\geq \frac{1}{2}\sint \eta^2\phi_B^2dx- \sint \eta^2 (\phi_1-\phi_B)^2dx \geq \frac{1}{2}\sint \eta^2\phi_B^2dx -C\varepsilon.
\end{equation}
We now focus our attention on $\sint \eta^2\phi_B^2$, and rewrite the integral as follows
\begin{equation*}
    \sint \eta^2\phi_B^2=\sint \phi_B^2+\sint (\eta^2-1)\phi_B^2 = 1 + \sint (\eta^2-1)\phi_B^2.
\end{equation*}
Since $\mathrm{spt}\, (\eta^2-1)\subseteq \mathbb{S}^{n-1}\setminus H_{2\varepsilon^{1/3}}$ and $\eta^2-1\geq -1$, we have
\begin{align*}
    \sint \eta^2\phi_B^2dx
    &= 1+\int_{\mathbb{S}^{n-1}\setminus H_{2\varepsilon^{1/3}}} (\eta^2-1)\phi_B^2dx\\
    &\geq 1-\int_{\mathbb{S}^{n-1}\setminus H_{2\varepsilon^{1/3}}} \phi_B^2dx\\
    &\geq 1-C\varepsilon^{1/3}(2\varepsilon^{1/3}+\varepsilon^{1/2})^2
\end{align*}
where the last inequality follows from the same reasoning as in \eqref{e: phiB bound}. Putting the above together with \eqref{e: denom bound 1} yields
\begin{equation*}
    \sint \eta^2 \phi_1^2 \geq  \frac{1}{2}\sint \eta^2\phi_B^2 - C\varepsilon \geq \frac{1}{4}-C\varepsilon \geq 1/8
\end{equation*}
for $\varepsilon$ small enough. Thus from \eqref{e: num bound} and the hypothesis on $\lambda_1(\Omega)$,
\begin{equation*}
    \lambda_1(H\setminus A)\leq \lambda_1(\Omega)+ C\varepsilon^{1/3} \leq n-1 + C\varepsilon^{1/3},
\end{equation*}
and moreover, from \eqref{e:001},  
\begin{equation*}\label{e: Cap bound 2}
    \mathrm{Cap}_{H}(A)\leq C\varepsilon^{1/3},
\end{equation*}
which concludes the proof.
\end{proof}

\begin{remark}
    If $\lambda(x,r)<\varepsilon$, then the hypotheses of Lemma \ref{lem: no sets of pos capacity far from BAHP} are satisfied for both $\Omega_1$ and $\Omega_2$, i.e. $|\lambda(\Omega_i)-(n-1)|<\varepsilon$ for $i=1,2$. Then the hemisphere $H$ chosen in the proof of Lemma \ref{lem: no sets of pos capacity far from BAHP} can be taken to be the hemisphere that attains the infimum for $\varepsilon_n(x,r)$. In particular, $\mathrm{Cap}_{\overline{H}^c}(\overline{\Omega}\cap \overline{H}^c_{\varepsilon^{\beta}})<C\varepsilon^{1/3}$.
\end{remark}

Lemma \ref{lem: no sets of pos capacity far from BAHP} says that both $\Omega$ and $(\mathbb{S}^{n-1}\setminus \overline{\Omega})$ contain spherical caps, $H_{\varepsilon^{\beta}}$ and $(\mathbb{S}^{n-1}\setminus \overline{H})_{\varepsilon^{\beta}}$, respectively, \textit{up to a set of small capacity}. But we aim to obtain a stronger result \textemdash that these caps are fully contained in the complementary domains, which would immediately give control on the $b\beta_{\infty}^s$-numbers. As mentioned above, it is not true that, fixing $r, x_0$, we have $\lambda(x_0, r) < Cr^\alpha \Rightarrow b\beta^s(x_0,r) \leq Cr^{\alpha'}$ (see Example \ref{ex:lambdavsbeta}). So in order to bound the $b\beta_{\infty}^s$-numbers it is crucial to use the decay of the $\lambda$-function at nearby points and scales. In fact, this decay also yields a topological result about the domain, which is extremely useful in proving control on the $b\beta_{\infty}^s$s.

\begin{definition}\label{d:corkscrew}
    We say that $\Omega^+$ is a two-sided corkscrew domain if there exists $M > 1$ and $r_0 > 0$ such that for every $x\in \partial \Omega$ and $r < r_0$ there exist points $p^{\pm}_r(x) \in \Omega^{\pm} \cap B(x, r)$ (where $\Omega^- = \overline{\Omega}^c$) such that $\mathrm{dist}(p^{\pm}_r(x), \partial \Omega) \geq M^{-1}r$.
\end{definition}

That these corkscrews exist is the content of the following lemma, which may be of independent interest: 
\begin{lemma}\label{lem:existence of corkscrews}
    Suppose $\Omega_1,\Omega_2\subset \mathbb{R}^n$ are open and disjoint domains and let $\Gamma=\partial \Omega_1\cup \partial \Omega_2$. If there exists an $r_0>0, C > 0$ such that 
    \begin{equation}\label{e:lambda assumption}
        \varepsilon_n(x,r)^2\leq Cr^{\alpha}
    \end{equation}
    for all $x\in \Gamma$ and $0<r\leq r_0$, then $\Omega_2=\mathbb{R}^n\setminus \overline{\Omega_1}$ and $\Omega_1$ is a two-sided corkscrew domain, with constants depending on $\alpha, r_0, C$ and $n$. 
\end{lemma}

\begin{remark}
    We remark that Lemma 3.2 holds under the slightly weaker assumption that $\varepsilon_n(x,t)^2 \leq \omega(t)$ where $\omega(t)$ is a modulus of continuity satisfying the integrability condition $$\int_0^1 \frac{\omega(t)}{t}\, dt < \infty.$$ Indeed, the proof below works essentially unchanged. 
\end{remark}

To prove this lemma, we recall a measure theoretic notion of corkscrew balls:
\begin{definition}[Definition 4.1, \cite{FTV23}]
    Let $\Omega\subset \mathbb{R}^n$, $B\subset \mathbb{R}^n$ be a ball, and $\gamma\in (0,1)$. Then a ball $B^{\prime}\subset \mathbb{R}^n$ is a \emph{$\gamma$-almost corkscrew ball for $B$ and $\Omega$} if $B^{\prime}\subset B$, $r(B^{\prime})\approx r(B)$, and 
    \begin{equation}\label{e:beta almost corkscrew}
 \mathcal{H}^{n}(B^{\prime}\setminus \Omega)\leq \gamma\,\mathcal{H}^n(B^{\prime}).
    \end{equation}
\end{definition}
Observe that if $B^{\prime}$ was a true corkscrew ball for $B$ and $\Omega$, \eqref{e:beta almost corkscrew} holds with $\gamma=0$ (in fact, $B^{\prime}\setminus \Omega=\emptyset$). It was shown that uniform decay of the $\varepsilon_n$ implies the existence of almost corkscrew balls in the following way:

\begin{lemma}[Lemma 5.3, \cite{FTV23}]\label{l:quasi-cork}
    Let $x,y\in \mathbb{R}^n$ such that $|x-y|=R$, for some $R$ small enough. There exists $\delta_0>0$ small enough so that if 
    \begin{equation}\label{e:eps cond quasi-cork}
        \int_0^{8R}\left(\varepsilon_n(x,t)^2+\varepsilon_n(y,t)^2 \right)\frac{dt}{t}\leq \delta.
    \end{equation}
for $0<\delta\leq \delta_0$, then there exist $B^1,B^{2}\subset B(x,6R)$ and $C>0$ such that each $B^i$ is a $C\delta^{1/4}$-almost corkscrew for $B(x,6R)$ and $\Omega_i$, with $r(B^i)\approx R$, for $i=1,2$.
\end{lemma}

We should note that we have rewritten the statement of \cite[Lemma 5.3]{FTV23} to highlight the dependence of the corkscrew constant on $\delta$ (this dependence follows immediately from their proof, see e.g., the last line of the proof of \cite[Lemma 5.3]{FTV23}).

We remark that the conclusion, $\Omega_2=\mathbb{R}^n\setminus\overline{\Omega_1}$ in Lemma \ref{lem:existence of corkscrews}, follows from arguments in \cite[Proposition 3.1]{CTV25}, with minor adjustments to account for the lack of Ahlfors regularity of $\partial \Omega_1\cup \partial \Omega_2$ and the different decay on $\varepsilon_n$ (as in \eqref{e:lambda assumption}). We provide a self-contained proof of this conclusion in the proof of Lemma \ref{lem:existence of corkscrews}. 

\begin{proof}[Proof of Lemma \ref{lem:existence of corkscrews}]  
We first show that $\Omega_2:=\mathbb{R}^n\setminus \overline{\Omega_1}$. We begin by observing that $\partial \Omega _1 = \partial \Omega_2$. Indeed, if there exists $x\in \partial \Omega_1 \setminus \partial \Omega_2$ with $d(x,\partial \Omega_2)>0$, then $\varepsilon_n(x,s)\approx 1$ for all $0<s<\frac{d(x,\partial \Omega_2)}{2}$, which contradicts \eqref{e:lambda assumption}. We now show $\Omega_2=\mathbb{R}^n\setminus \overline{\Omega_1}$. Assume not and let $x_0 \in \mathbb R^n\backslash (\overline{\Omega}_1\cup \overline{\Omega}_2)$. Let $r>0$ be such that $B(x_0,r) \subset \mathbb R^n\backslash (\overline{\Omega}_1\cup \overline{\Omega}_2)$ but there exists a $y\in \overline{B(x_0, r)}\cap (\overline{\Omega}_1\cup \overline{\Omega}_2)$. Note by translating $x_0$ if necessary, we can take $r\leq r_0$. Then note that $$\mathcal H^{n-1}(\partial B(y, r)\backslash (\overline{\Omega_1}\cup \overline{\Omega_2})) \geq \mathcal H^{n-1}(\partial B(y,r)\cap B(x_0, r/2)) \geq c_nr^{n-1}.$$
 This contradicts the decay of $\varepsilon_n(y,r)$.
 Thus we rename 
$$
\Omega^+:=\Omega_1\qquad \text{and} \qquad \Omega^-:=\Omega_2.
$$
It is left to show that $\Omega^+$ is a two-sided corkscrew open set.

Fix $x\in \partial \Omega^+$. Observe that for any $0<R\leq r_0$ there exists $y\in \partial \Omega^+\cap \partial B(x,R)$, otherwise $\varepsilon_n(x,R)\approx \sigma_n$, which would contradict \eqref{e:lambda assumption}. Thus, \eqref{e:eps cond quasi-cork} holds with $\delta=C(\alpha)R^{\alpha}$, for $0<R\leq r_0$. It then follows from Lemma \ref{l:quasi-cork} that there exists $B^{+},B^-\subset B(x,6R)$ such that $B^{+},B^-$ are $(C\delta)^{1/4}$-almost corkscrew balls for $B(x,6R)$ and $\Omega^{+},\Omega^-$, respectively, where $r(B^{i})\approx R$, for $i=+,-$. 
We claim that $\frac{1}{2}B^{+}\subset \Omega^{+}$. The analogous statement for $\frac{1}{2}B^-$ follows from the below arguments, replacing $B^+$ and $\Omega^+$ with $B^-$ and $\Omega^-$, respectively. 

Suppose for contradiction that there exists $x_0\in \frac{1}{2}B^+\cap \partial \Omega^+$, then $B(x_0,\frac{1}{2}r(B^+))\subseteq B^+$. Now, $\varepsilon_n(x_0,s)\leq Cs^{\alpha}$ for all $0<s<\frac{r(B^+)}{2}$, which implies (because $r_0$ is chosen small enough) that 
\begin{equation}\label{e:balanceonsphere}
    \frac{1}{4}\sigma_ns^{n-1}\leq \mathcal{H}^{n-1}(\partial B(x_0,s)\cap \Omega^-).
\end{equation}
Thus, from \eqref{e:beta almost corkscrew}, \eqref{e:balanceonsphere} and the co-area formula
\begin{align*}
    CR^{\alpha/4}\mathcal{H}^n(B^+)
    &\geq \mathcal{H}^n(B^+\cap \Omega^-)\\
    &\geq \mathcal{H}^n(B(x_0,\frac{1}{2}r(B^+))\cap \Omega^-)\\
    &= \int_0^{\frac{1}{2}r(B^+)} \mathcal{H}^{n-1}(\partial B(x_0,\frac{1}{2}r(B^+))\cap \Omega^-)ds\\
    &\geq C(n)\int_0^{\frac{1}{2}r(B^+)}s^{n-1}ds\\
    &=C(n)\mathcal{H}^n(B^+),
\end{align*}
where $C(n)$ is some dimensional constant. Thus, if $r_0$ is initially taken small enough depending on dimension, this is a contradiction and we have our corkscrew ball. 
\end{proof}

Before proving Lemma \ref{lem:no sets of zero capactiy far from BAHP}, it will be useful to know that there are points in $\partial \Omega^+$ close to the boundary of the best approximating hemisphere.

\begin{lemma}\label{lem: boundary close to hemisphere boundary}
    Let $\varepsilon\in (0,1)$ small. Suppose $V\subseteq \mathbb{S}^{n-1}$ and suppose there exists a hemisphere $H_V$ such that $$\mathcal H^{n-1}\left[\left(H_V\backslash V\right)\cup \left(H_V^c \backslash \overline{V}^c\right)\right] \leq \varepsilon.$$ Then there exists $\eta_0>0$ depending on $n,\varepsilon$ so that for any $p\in \partial H_V$ and $\eta\leq \eta_0$ there exists a point $q\in \partial V$ such that $|q-p|<\varepsilon^{\eta}$. 
\end{lemma}

\begin{proof}
Assume the geodesic ball $B(p, \varepsilon^{\eta})$ on $\mathbb{S}^{n-1}$ is either entirely contained in $A$ or entirely contained in $\overline{A}^c$. Either way this implies that $\varepsilon^{\eta(n-1)} \leq C_n \varepsilon$, which yields a contradiction if $\eta\lesssim_{n,\varepsilon}\frac{1}{n-1}$. Thus we have the desired result. 
\end{proof}

\begin{lemma}[Control on the $b\beta_{\infty}^s$s]\label{lem:no sets of zero capactiy far from BAHP}
Let $\Omega^+\subseteq \mathbb{R}^n$ be a two-sided corkscrew open set and let $\Omega^-:=\mathbb{R}^n\setminus \overline{\Omega^+}$ and $\Gamma:= \partial \Omega^{+}$. Fix $y_0\in \Gamma$. Suppose that for all $x\in B(y_0,2)\cap \Gamma$ and $1/2 < r < 2$ we have $\lambda(x, r) \leq \varepsilon < \varepsilon_0$. Let $\gamma < \frac{1}{200(n+1)}$. Then if $\varepsilon_0$ is small enough (depending on $\gamma, n$ and the corkscrew constants), for each $x_0\in B(y_0,1)$ there are complementary half-spheres $H^{\pm}(x_0, 1)\subset \partial B(x_0, 1)$ such that

$$\sup_{y\in \overline{\Omega^{\pm}} \cap \partial B(x_0,1)} \mathrm{dist}(y, H^{\pm}(x_0, 1))\leq C\varepsilon^{\gamma}.$$
Moreover, 
$$\begin{aligned} \{y\in \partial B(x_0,1)\mid (y-x) \cdot e_n \geq \varepsilon^{\gamma}\} \subset& \Omega^+, \qquad \text{and}\\
    \{y\in \partial B(x_0,1)\mid (y-x) \cdot e_n \leq -\varepsilon^{\gamma}\} \subset& \Omega^-. 
    \end{aligned}$$

\end{lemma}
\begin{proof}
Fix $x_0\in \Gamma$ and consider $\partial B(x_0,1)$. Assume by rotating and translating if necessary that $x_0=0$. Moreover, assume $H^+(0,1)=\{x_n>0\}\cap \partial B(0,1)$, where $H^{+}$ is chosen as in Lemma \ref{lem: no sets of pos capacity far from BAHP}, and for the rest of the proof we let $H:=H^+$.

We assume for contradiction that there exists 
$$
y\in \Gamma \cap \partial B(0,1) \qquad \text{with}\qquad d(y,\partial H)>\varepsilon^{\gamma}.
$$
Suppose without loss of generality $y\in H$, and let $B(p^-,\frac{1}{100M}\varepsilon^{\gamma})$ denote the corkscrew ball for $\Omega^-$ and $B(y,\frac{1}{100}\varepsilon^{\gamma})$, where $M>1$.  See Figure \ref{f:set up}. Note that if instead $y\in \overline{H}^c$, then instead we would consider $B(p^+,\frac{1}{100M}\varepsilon^{\gamma})$, the corkscrew ball for $\Omega^+$ and $B(y,\frac{1}{100}\varepsilon^{\gamma})$. 

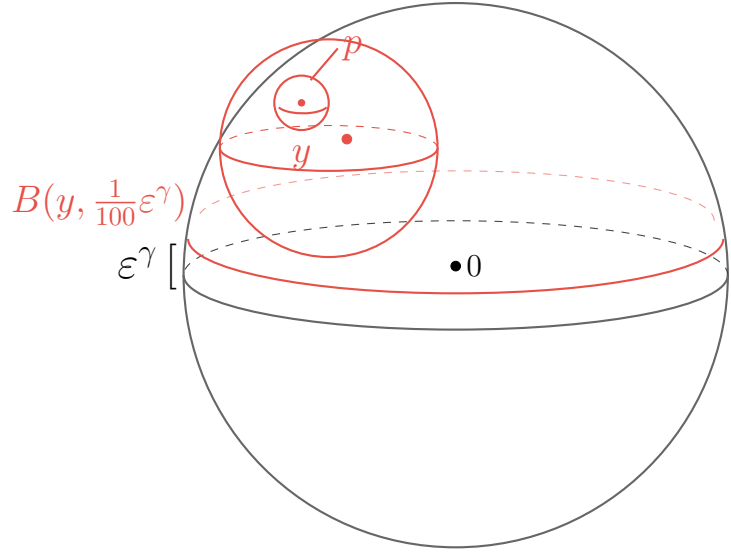
\begin{figure}
    \centering
\begin{tikzpicture}[scale=1.2, >=stealth]

    \definecolor{reddish}{RGB}{230, 80, 70}

    \draw[thick, gray!80!black] (0,0) circle (3cm);

    \draw[thick, gray!80!black] (-3,0) arc (180:360:3cm and 0.6cm);
    \draw[gray!50!black, dashed] (3,0) arc (0:180:3cm and 0.6cm);

    \filldraw[black] (0,0.1) circle (1.5pt) node[right=2pt] {$0$};

    \draw[thick, reddish] (-2.95,0.4) arc (180:360:2.95cm and 0.6cm);
    \draw[dashed, reddish,  opacity=0.6] (2.85,.6) arc (0:180:2.85cm and 0.55cm);

    \node[ font=\Large] at (-3.4,0.1) {$\varepsilon^\gamma \, [$};

    \begin{scope}[shift={(-1.4,1.4)}]
        \draw[thick, reddish] (0,0) circle (1.2cm);

        \draw[thick, reddish] (-1.2,0) arc (180:360:1.2cm and 0.25cm);
        \draw[reddish, dashed] (1.2,0) arc (0:180:1.2cm and 0.25cm);

        \filldraw[reddish] (0.2,0.1) circle (1.5pt);
        \node[reddish, font=\large] at (-0.3,-0.1) {$y$};

        \draw[thick, reddish] (-0.3,0.5) circle (0.3cm);
        \draw[thick, reddish] (-0.55,0.45) arc (200:340:0.28cm and 0.1cm);
        \filldraw[reddish] (-0.3,0.5) circle (1pt);

        \draw[thick, reddish] (-0.2,0.75) -- (0.1,1.1) node[right, font=\large] {$p$};

        \node[reddish, font=\large, left] at (-1.5,-0.6) {$B(y, \frac{1}{100}\varepsilon^\gamma)$};
    \end{scope}

\end{tikzpicture}
    
    \caption{The set up}
    \label{f:set up}
\end{figure}

Let $z\in \partial H$ be chosen so that $|y-z|=\sqrt{2}$. Then from Lemma \ref{lem: boundary close to hemisphere boundary} it follows that there exists $q\in \partial B(0,1) \cap \Gamma$ such that 
\begin{equation*}
    |z-q|<\varepsilon^{16\gamma}.
\end{equation*}
 Let $R=|p^--q|$ and note that $R\approx \sqrt{2}$. In more details,
\begin{align*}
    R&\leq |p^--y|+|y-z|+|z-q|\leq \frac{1}{100}\varepsilon^{\gamma}+\sqrt{2}+\varepsilon^{2\gamma}\leq \frac{1001}{1000}\sqrt{2},\\
    R&\geq |y-z|-|p^--y|-|z-q|\geq \sqrt{2}- \frac{1}{100}\varepsilon^{\gamma}+-\varepsilon^{2\gamma}\geq \frac{999}{1000}\sqrt{2},\\
\end{align*}
for $\varepsilon$ small enough depending on $\gamma$. Now consider $B(q,R)$. 

\begin{figure}
    \centering
\begin{tikzpicture}[scale=1.2, >=stealth]

    \definecolor{reddish}{RGB}{230, 80, 70}
    \definecolor{softblue}{RGB}{60, 150, 245}

    \fill[reddish, opacity=0.7] (0.8, 3.8) rectangle (1.0, 4.0);
    \node[reddish, font=\small, right] at (1.0, 3.9) {$= B(p, \frac{1}{1000}\varepsilon^\gamma) \cap \partial B(q, R)$};

    \draw[thick, gray!80!black] (0,0) circle (3cm);

    \draw[thick, gray!80!black] (-3,0) arc (180:360:3cm and 0.6cm);
    \draw[gray!50!black, dashed] (3,0) arc (0:180:3cm and 0.6cm);

    \draw[thick, softblue!80!white] (-0.5, 0.05) circle (3.02cm);
    \node[softblue, font=\large, right] at (-4.95, -1.8) {$B(q,R)$};

    \draw[very thick, softblue] (0.1, 3.0) arc (90:-90:0.8cm and 3.0cm);
    \draw[thick, softblue!60!white, dashed] (0.1, 3.0) arc (90:270:0.8cm and 3.0cm);
    
    \node[softblue, font=\large, below] at (0.1, -3.0) {$\Sigma_R$};

    \filldraw[black] (0.3, 0.2) circle (1.5pt) node[right=2pt] {$0$};

    \draw[thick, reddish] (-2.95, 0.4) arc (180:360:2.95cm and 0.6cm);
    \draw[thick, reddish, dashed] (2.85, .70) arc (0:180:2.85cm and 0.55cm);

    \node[reddish, font=\Large] at (3.4, 0.15) {$]\,\varepsilon^\gamma$};

    \begin{scope}[shift={(-.68,-0.6)}]
        \draw[thick, black] (0,0) circle (0.5cm);
        
        \filldraw[black] (-0.15,0.1) circle (1.5pt) node[left=1pt, font=\small] {$q$};
        \filldraw[black] (0.1,0) circle (1.5pt) node[below=1pt, font=\small] {$z$};

        \node[black, font=\small, below left] at (0.30,-0.47) {$B(z, \varepsilon^{2\gamma})$};
    \end{scope}

    \begin{scope}[shift={(-1.9,1.7)}]
        \draw[thick, reddish] (0,0) circle (1.2cm);

        \draw[thick, reddish] (-1.2,0) arc (180:360:1.2cm and 0.25cm);
        \draw[reddish, dashed] (1.2,0) arc (0:180:1.2cm and 0.25cm);

        \filldraw[reddish] (0.2,-0.1) circle (1.5pt);
        \node[reddish, font=\large] at (0.35,-0.15) {$y$};

        \begin{scope}[shift={(-0.27, 0.58)}]
            \draw[thick, reddish] (0,0) circle (0.3cm);
            \draw[thick, reddish] (-0.25,-0.05) arc (200:340:0.28cm and 0.1cm);
            
            \fill[reddish, opacity=0.7] (0,0) circle (0.2cm);
            \filldraw[reddish] (0,0) circle (1.5pt);
        \end{scope}

        \draw[thick, reddish] (-0.25, 0.95) -- (-0.1, 1.3) node[right, font=\large] {$p$};

        \node[reddish, font=\large, left] at (-1.35,-0.6) {$B(y, \frac{1}{100}\varepsilon^\gamma)$};
    \end{scope}

\end{tikzpicture}
    \caption{The set up: $B(q,R)$}
    \label{fig:B(q,R)}
\end{figure}
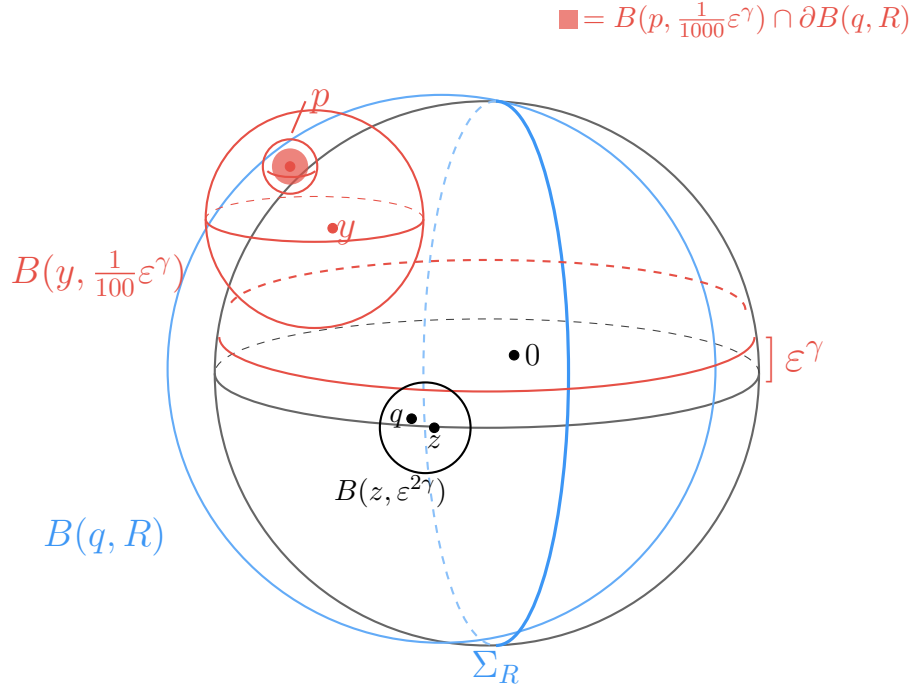

We will derive a contradiction on $\partial B(q,R)$. We first provide a heuristic for the argument. We begin by using the fact that $\partial B(q,R)$ has large intersection with $\Omega^-$ near $p^-$, and Proposition \ref{p:FTV thrm c} to deduce that $H^+(q,R)$ has a large intersection with $\overline{H}^c$. In fact, we show $H^+(q,R)\cap \overline{H}^c$ contains a large $\mathcal{H}^{n-2}$-measure subset of $\Sigma_R:=\partial B(0,1)\cap \partial B(q,R)$. Let's call $A\subset \Sigma_R$ the large portion of $\Sigma_R$ that is contained in $H^+(q,R)\cap \overline{H}^c$ and also far from their boundaries, i.e. $d(A,\partial H),\,d(A,\partial H^+(q,R))$ are big. Then, from Lemma \ref{lem: no sets of pos capacity far from BAHP}, the capacity of $A\cap \overline{\Omega^j}$ should be small for $j=+,-$, but, for some choice of $+,-$, this is a contradiction with Lemma \ref{lem: relative cap to cap}. We now make this precise.

Let $H_R:=H^+(q,R)$ denote the hemisphere on $\partial B(q,r)$ such that 
\begin{equation*}
    \varepsilon_n(q,R)\leq \frac{1}{R^{n-1}} \mathcal{H}^{n-1}\left([H_R\setminus \Omega^+] \cup [\overline{H_R}^c \setminus \Omega^-] \right)\leq \varepsilon^{1/2},
\end{equation*}
and since $R\approx \sqrt{2}$, 
\begin{equation}\label{e:eps for q}
    \mathcal{H}^{n-1}\left([H_R\setminus \Omega^+] \cup [\overline{H_R}^c \setminus \Omega^-] \right)\lesssim_n \varepsilon^{1/2}.
\end{equation}

 Note that $H_R=\{w\in \partial B(q,R):\, (w-q)\cdot e_R>0\}$, for some $e_R\in \mathbb{S}^{n-1}$.

 We let $\Sigma_R:=\partial B(0,1)\cap \partial B(q,R)$, which is an $(n-2)$-sphere with center $tq$, where $t=1-\frac{R^2}{2}$ and where $\rho_R\approx R$, where $\rho_R$ denotes the radius of $\Sigma_R$.

 \noindent\textbf{Step 1:} Quantify the rotation between $H$ and $H_R$.

We first show that $(p^--q)\cdot e_R\leq\frac{2}{100M}\varepsilon^{\gamma}$. If not  
then for $z\in B(p^-,\frac{1}{100M}\varepsilon^{\gamma})\cap \partial B(q,R)$, 
\begin{align*}
    (z-q)\cdot e_R=(z-p^-)\cdot e_R+(p^--q)\cdot e_R\geq -\frac{1}{100M}\varepsilon^{\gamma}+\frac{2}{100M}\varepsilon^{\gamma}>0.
\end{align*}
Thus, $B(p^-,\frac{1}{100M}\varepsilon^{\gamma})\cap \partial B(q,R)\subset H_R\cap \Omega^-$, which contradicts \eqref{e:eps for q}. Thus $(p^--q)\cdot e_R<\frac{2}{100M}\varepsilon^{\gamma}$. Moreover, since $y\in H$, $d(y,\partial H)>\varepsilon^{\gamma}$ and $p^-\in B(y,\frac{1}{100}\varepsilon^{\gamma})$, we have $p^-\cdot e_n\geq C\varepsilon^{\gamma}$.

Thus, 
\begin{align*}
    2e_n\cdot e_R&=2-|e_n-e_R|^2\\
    &\leq 2-\left(|e_R-p^-|+|p^--e_n|\right)^2\\
    &\leq 2-\left(\frac{\pi}{2}-\frac{2}{100M}\varepsilon^{\gamma}-\frac{\pi}{2}+\varepsilon^{\gamma}\right)^2, 
\end{align*}
which gives 
\begin{equation}\label{e: e_q e_n}
    e_n\cdot e_R \leq 1-C\varepsilon^{2\gamma}.
\end{equation}

 Moreover, by construction of $q$ and the triangle inequality, 
\begin{equation}\label{e:q e_n}
    |q\cdot e_n| = |(q-z)\cdot e_n| \leq \varepsilon^{16\gamma}
\end{equation}

\noindent\textbf{Step 2:} We now find a point $w\in \Sigma_R$ such that either $w\in \overline{H_R}^c\cap H$ or  $w\in H_R\cap \overline{H}^c$, and 
\begin{equation}\label{e:wdef}
    d(w,\partial H)\gtrsim \varepsilon^{8\gamma} \qquad \text{and} \qquad d(w,\partial H_R)\gtrsim \varepsilon^{8\gamma}.
\end{equation}

We first set up some notation. Let $v\in \mathbb{S}^{n-1}$ such that $v\cdot q=0$ and 
\begin{equation}\label{e: v e_n}
   v\cdot e_n \geq 1-\frac{1}{2}\varepsilon^{32\gamma}.  
\end{equation}

Such $v$ exists by \eqref{e:q e_n}. Let $E:=\{e\in \mathbb{S}^{n-1}:\, e\cdot q=e\cdot v=0\}$. Thus, there exists $a,b,c\in \mathbb{R}$ such that 
$$
e_R=aq+bv+ce, \qquad \text{where} \quad a^2+b^2+c^2=1,\,\, \text{and}\,\, e\in E.
$$
Observe that for any $e\in E$, since $ e\cdot e_n= e\cdot (e_n-v)$, where 
\begin{align*}
    |e\cdot (e_n-v)|&\leq |e_n-v|=\sqrt{2-2e_n\cdot v}\leq \sqrt{2-2(1-\frac{1}{2}\varepsilon^{32\gamma})}=\varepsilon^{16\gamma},
\end{align*}
we have 
\begin{equation}\label{e: e e_n}
    -\varepsilon^{16\gamma}\leq e\cdot e_n\leq \varepsilon^{16\gamma}.
\end{equation}
Moreover, using that $|e_n-v|\leq \varepsilon^{16\gamma}$, which is computed above, we obtain
\begin{equation}\label{e: b approx e_r cdot e_n}
e_R\cdot e_n-\varepsilon^{16\gamma}\leq e_R\cdot v\leq e_R\cdot e_n + \varepsilon^{16\gamma}.
\end{equation}

We will write our desired $w\in \Sigma_R$ satisfying \eqref{e:wdef} as 
\begin{equation*}
    w=tq\pm d\rho_Rv\pm\sqrt{1-d^2} \,\rho_Re, 
\end{equation*}
for some choice of $+$ and $-$, and some $0< d\leq \frac{1}{\rho_R}$. These choices depend on relationship of $e_R$ to $e_n$ and $q$. Thus we resort to case analysis.

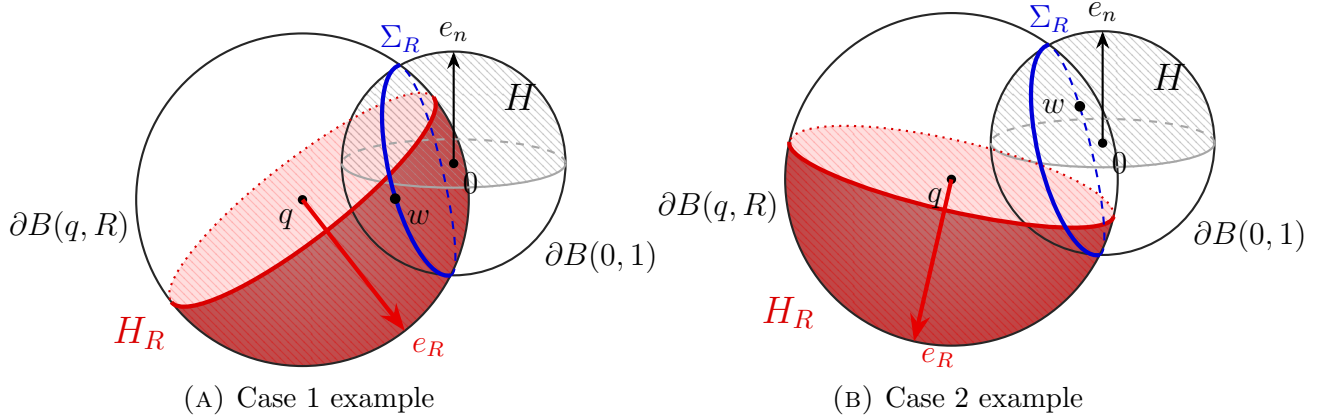
\begin{figure}[htbp]
    \centering
    \begin{subfigure}[b]{0.48\textwidth}
        \centering
        \begin{tikzpicture}[>=Stealth, scale=.8]

  \coordinate (O) at (1.2, 0.35);
  \def\RO{1.85}

  \coordinate (Q) at (-1.3, -0.25);
  \def\RQ{2.75}

  \def\angleER{-52}
  \pgfmathsetmacro{\angleERPerp}{\angleER + 90}

  \path let \p1 = ($(Q)-(O)$) in 
    \pgfextra{
      \pgfmathsetmacro{\angleQ}{atan2(\y1,\x1)} 
      \pgfmathsetmacro{\angleQPerp}{\angleQ + 90}
      \global\let\angleQ\angleQ
      \global\let\angleQPerp\angleQPerp
    };

  \path[name path=sphereQ] (Q) circle (\RQ);
  \path[name path=sphereO] (O) circle (\RO);

  \path[name intersections={of=sphereQ and sphereO, by={P1, P2}}];
  
  \coordinate (M) at ($(P1)!0.5!(P2)$);


  \begin{scope}
    \clip (Q) circle (\RQ);
    \begin{scope}[shift={(Q)}, rotate=\angleERPerp]
      \clip (\RQ, 0) arc (0:180:\RQ\space and 0.65) -- (-\RQ, -3.5) -- (\RQ, -3.5) -- cycle;
      \fill[red!20, opacity=0.65] (0,0) circle (\RQ);
      \fill[pattern=north west lines, pattern color=red!45, opacity=0.7] (0,0) circle (\RQ);
    \end{scope}
  \end{scope}

  \begin{scope}
    \clip (Q) circle (\RQ);
    \begin{scope}[shift={(Q)}, rotate=\angleERPerp]
      \clip (-\RQ, 0) arc (180:360:\RQ\space and 0.65) -- (\RQ, -3.5) -- (-\RQ, -3.5) -- cycle;
      \fill[shade, shading=radial, inner color=red!40, outer color=red!80, opacity=0.6] (0,0) circle (\RQ);
      \fill[pattern=north west lines, pattern color=red!90!black, opacity=0.55] (0,0) circle (\RQ);
    \end{scope}
  \end{scope}

  \begin{scope}
    \clip (O) circle (\RO);
    \begin{scope}[shift={(O)}]

    \clip (-\RO, 0) arc (180:360:\RO\space and {.23*\RO}) -- (\RO, \RO) -- (-\RO, \RO) -- cycle;
      \fill[pattern=north west lines, pattern color=black!35, opacity=0.75] (0,0) circle (\RO);
    \end{scope}
  \end{scope}

  \draw[thick, black!85] (Q) circle (\RQ);
  \draw[thick, black!85] (O) circle (\RO);

  \draw[gray!70, thick] (O) ++(-\RO, 0) arc (180:360:\RO\space and 0.4);
  \draw[gray!70, dashed, thick] (O) ++(\RO, 0) arc (0:180:\RO\space and 0.4);

  \begin{scope}[shift={(Q)}, rotate=\angleERPerp]
    \draw[red!85!black, ultra thick] (-\RQ, 0) arc (180:360:\RQ\space and 0.65);
    \draw[red!85!black, thick, dotted] (\RQ, 0) arc (0:180:\RQ\space and 0.65);
  \end{scope}

  \path let \p1 = ($(P1)-(M)$), \n1 = {veclen(\x1,\y1)} in
    \pgfextra{
      \pgfmathsetmacro{\radSigma}{\n1/1cm}
      \pgfmathsetmacro{\minorSigma}{0.42 * \radSigma}
      \global\let\radSigma\radSigma
      \global\let\minorSigma\minorSigma
    };

  \begin{scope}[shift={(M)}, rotate=\angleQPerp]
    \draw[blue!85!black, ultra thick] (\radSigma, 0) arc (0:-180:{.9*\radSigma\space+.18} and {.8*\minorSigma});
    \draw[blue!85!black, thick, dashed] (\radSigma, 0) arc (0:180:{.9*\radSigma\space+.15} and {.9*\minorSigma-.4});
  \end{scope}


  \fill[black] (Q) circle (2.2pt) node[below left=2pt] {$q$};
  \fill[black] (O) circle (2.2pt) node[below right=2pt] {$0$};
  \fill[black] ($(M)+(-.5,-.47)$) circle (2.5pt) node[below right=3.5pt] {$w$};

  \draw[->, thick, >=Stealth] (O) -- ++(0, 1.85) node[above=2pt] {$e_n$};

  \begin{scope}[shift={(Q)}]
    \draw[->, red!90!black, ultra thick, >=Stealth] (0,0) -- (\angleER:\RQ) node[below right=1pt] {$e_R$};
  \end{scope}

  \node[blue!85!black, above=3pt] at (P1) {$\Sigma_R$};
  \node[font=\large] at ($(O) + (1.1, 1.1)$) {$H$};
  \node[red!85!black, font=\large] at ($(Q) + (-2.7, -2.2)$) {$H_R$};

  \node[below left] at ($(Q) +(-1,2) + (-130:\RQ)$) {$\partial B(q,R)$};
  \node[below right] at ($(O) + (-40:\RO)$) {$\partial B(0,1)$};

\end{tikzpicture}
        \caption{Case 1 example}
        \label{fig:case 1}
    \end{subfigure}
    \hfill 
    \begin{subfigure}[b]{0.48\textwidth}
        \centering
         \begin{tikzpicture}[>=Stealth, scale=.8]

  \coordinate (O) at (1.2, 0.35);
  \def\RO{1.85}

  \coordinate (Q) at (-1.3, -0.25);
  \def\RQ{2.75}

  \def\angleER{-103}
  \pgfmathsetmacro{\angleERPerp}{\angleER + 90}

  \path let \p1 = ($(Q)-(O)$) in 
    \pgfextra{
      \pgfmathsetmacro{\angleQ}{atan2(\y1,\x1)} 
      \pgfmathsetmacro{\angleQPerp}{\angleQ + 90}
      \global\let\angleQ\angleQ
      \global\let\angleQPerp\angleQPerp
    };

  \path[name path=sphereQ] (Q) circle (\RQ);
  \path[name path=sphereO] (O) circle (\RO);

  \path[name intersections={of=sphereQ and sphereO, by={P1, P2}}];
  
  \coordinate (M) at ($(P1)!0.5!(P2)$);


  \begin{scope}
    \clip (Q) circle (\RQ);
    \begin{scope}[shift={(Q)}, rotate=\angleERPerp]
      \clip (\RQ, 0) arc (0:180:\RQ\space and 0.65) -- (-\RQ, -3.5) -- (\RQ, -3.5) -- cycle;
      \fill[red!20, opacity=0.65] (0,0) circle (\RQ);
      \fill[pattern=north west lines, pattern color=red!45, opacity=0.7] (0,0) circle (\RQ);
    \end{scope}
  \end{scope}

  \begin{scope}
    \clip (Q) circle (\RQ);
    \begin{scope}[shift={(Q)}, rotate=\angleERPerp]
      \clip (-\RQ, 0) arc (180:360:{.8*\RQ\space+.55} and {.8*0.65}) -- (\RQ, -3.5) -- (-\RQ, -3.5) -- cycle;
      \fill[shade, shading=radial, inner color=red!40, outer color=red!80, opacity=0.6] (0,0) circle (\RQ);
      \fill[pattern=north west lines, pattern color=red!90!black, opacity=0.55] (0,0) circle (\RQ);
    \end{scope}
  \end{scope}

  \begin{scope}
    \clip (O) circle (\RO);
    \begin{scope}[shift={(O)}]

    \clip (-\RO, 0) arc (180:360:\RO\space and {.23*\RO}) -- (\RO, \RO) -- (-\RO, \RO) -- cycle;
      \fill[pattern=north west lines, pattern color=black!35, opacity=0.75] (0,0) circle (\RO);
    \end{scope}
  \end{scope}

  \draw[thick, black!85] (Q) circle (\RQ);
  \draw[thick, black!85] (O) circle (\RO);

  \draw[gray!70, thick] (O) ++(-\RO, 0) arc (180:360:\RO\space and 0.4);
  \draw[gray!70, dashed, thick] (O) ++(\RO, 0) arc (0:180:\RO\space and 0.4);

  \begin{scope}[shift={(Q)}, rotate=\angleERPerp]
    \draw[red!85!black, ultra thick] (-\RQ, 0) arc (180:360:{.8*\RQ\space+.55} and {.8*0.65});
    \draw[red!85!black, thick, dotted] (\RQ, 0) arc (0:180:\RQ\space and 0.65);
  \end{scope}

  \path let \p1 = ($(P1)-(M)$), \n1 = {veclen(\x1,\y1)} in
    \pgfextra{
      \pgfmathsetmacro{\radSigma}{\n1/1cm}
      \pgfmathsetmacro{\minorSigma}{0.42 * \radSigma}
      \global\let\radSigma\radSigma
      \global\let\minorSigma\minorSigma
    };

  \begin{scope}[shift={(M)}, rotate=\angleQPerp]
    \draw[blue!85!black, ultra thick] (\radSigma, 0) arc (0:-180:{.9*\radSigma\space+.18} and {.8*\minorSigma});
    \draw[blue!85!black, thick, dashed] (\radSigma, 0) arc (0:180:{.9*\radSigma\space+.15} and {.9*\minorSigma-.4});
  \end{scope}


  \fill[black] (Q) circle (2.2pt) node[below left=2pt] {$q$};
  \fill[black] (O) circle (2.2pt) node[below right=2pt] {$0$};
  \fill[black] ($(M)+(.09,.72)$) circle (2.5pt) node[left=3.5pt] {$w$};

  \draw[->, thick, >=Stealth] (O) -- ++(0, 1.85) node[above=2pt] {$e_n$};

  \begin{scope}[shift={(Q)}]
    \draw[->, red!90!black, ultra thick, >=Stealth] (0,0) -- (\angleER:\RQ) node[below right=1pt] {$e_R$};
  \end{scope}

  \node[blue!85!black, above=3pt] at (P1) {$\Sigma_R$};
  \node[font=\large] at ($(O) + (1.1, 1.1)$) {$H$};
  \node[red!85!black, font=\large] at ($(Q) + (-2.7, -2.2)$) {$H_R$};

  \node[below left] at ($(Q) +(-1,2) + (-130:\RQ)$) {$\partial B(q,R)$};
  \node[below right] at ($(O) + (-40:\RO)$) {$\partial B(0,1)$};

\end{tikzpicture}
        \caption{Case 2 example}
        \label{fig:case 2}
    \end{subfigure}
    
    \caption{Examples for Cases 1 \& 2}
    \label{fig:cases12}
\end{figure}

\medskip

 \noindent\textbf{Case 1:} Suppose $q\cdot e_R \leq -\varepsilon^{4\gamma}$.\\
 Let $d=\varepsilon^{8\gamma}$, so, 
 $$w:=tq-\varepsilon^{8\gamma} \rho_Rv+\mathrm{sign}(c)\sqrt{1-\varepsilon^{16\gamma}}\rho_Re.$$ Note $\mathrm{sign}(c) = c/|c|$ unless $c = 0$ in which case we can choose $\mathrm{sign}(c) = \pm 1$ arbitrarily. 
 
 We first show that $w\in H_R$. Recall that $-R^2/2 = t-1<0$, then from \eqref{e: b approx e_r cdot e_n}, 
 \begin{align*}
     (w-q)\cdot e_R&= (t-1)q\cdot e_R-\varepsilon^{8\gamma} \rho_Rv\cdot e_R + \sqrt{1-\varepsilon^{16\gamma}}\rho_R|c|\\
     &\geq \frac{R^2}{2}\varepsilon^{4\gamma}-\rho_R\varepsilon^{8\gamma} \\
     &\geq \frac{1}{4}\rho_R\varepsilon^{4\gamma},
 \end{align*}
 for $\varepsilon$ small enough.

 Moreover, $w\in \overline{H}^c$. By \eqref{e:q e_n}, \eqref{e: v e_n}, and \eqref{e: e e_n},
 \begin{align*}
     w\cdot e_n&=tq\cdot e_n-\varepsilon^{8\gamma} \rho_R v\cdot e_n +\mathrm{sign}(c)\sqrt{1-\varepsilon^{16\gamma}}\rho_Re\cdot e_n\\
     &\leq t\varepsilon^{16\gamma}-\rho_R\varepsilon^{8\gamma}(1-\frac{1}{2}\varepsilon^{16\gamma})+\rho_R\varepsilon^{16\gamma}\\
     &\lesssim -\frac{1}{2}\rho_R\varepsilon^{8\gamma},
 \end{align*}
 for $\varepsilon$ small enough. This concludes \textbf{Case 1}.

 \vspace{\baselineskip}

 \noindent\textbf{Case 2:} Suppose $q\cdot e_R\geq \varepsilon^{4\gamma}$.\\
 
   Let $d=\varepsilon^{8\gamma}$, so 
  $$
  w:=tq+\rho_R\varepsilon^{8\gamma}v-\mathrm{sign}(c)\sqrt{1-\varepsilon^{16\gamma}}\, \rho_Re,
  $$ 
   where again $\mathrm{sign}(c) = c/|c|$, unless $c = 0$ in which case we can choose $\mathrm{sign}(c) = \pm 1$ arbitrarily.

  We now show that $w\in \overline{H_R}^c$. Indeed,
 \begin{align*}
     (w-q)\cdot e_R&= (t-1)q\cdot e_R+\rho_R\varepsilon^{8\gamma}v\cdot e_R-\mathrm{sign}(c)\sqrt{1-\varepsilon^{16\gamma}}\, \rho_Re\cdot e_R\\
     &\leq -\frac{R^2}{2}\varepsilon^{4\gamma}+\rho_R\varepsilon^{8\gamma}\\
      &< -\varepsilon^{4\gamma}\rho_R,
 \end{align*}
 for $\varepsilon$ small enough. Moreover, $w\in H$. From \eqref{e:q e_n}, \eqref{e: v e_n}, and \eqref{e: e e_n}:
 \begin{align*}
     w\cdot e_n
     &=tq\cdot e_n+\rho_R\varepsilon^{8\gamma}v\cdot e_n-\mathrm{sign}(c)\sqrt{1-\varepsilon^{16\gamma}}\rho_Re\cdot e_n\\
     &\geq -t\varepsilon^{16\gamma}+\rho_R\varepsilon^{8\gamma}(1-\frac{1}{2}\varepsilon^{32\gamma})-\sqrt{1-\varepsilon^{2\gamma}}\rho_R\varepsilon^{16\gamma}\\
     &> \frac{1}{2}\rho_R\varepsilon^{8\gamma}, 
 \end{align*}
for $\varepsilon$ small enough. This concludes \textbf{Case 2}.
 \vspace{\baselineskip}
 Observe that now in the remaining cases, we can assume $|q\cdot e_R|\leq \varepsilon^{4\gamma}$.\\

 \begin{figure}[htbp]
    \centering
    \begin{subfigure}[b]{0.48\textwidth}
        \centering
    \begin{tikzpicture}[>=Stealth, scale=.8]

  \coordinate (O) at (1.2, 0.35);
  \def\RO{1.85}

  \coordinate (Q) at (-1.3, -0.25);
  \def\RQ{2.75}

  \def\angleER{-45}
  \pgfmathsetmacro{\angleERPerp}{\angleER + 90}

  \path let \p1 = ($(Q)-(O)$) in 
    \pgfextra{
      \pgfmathsetmacro{\angleQ}{atan2(\y1,\x1)} 
      \pgfmathsetmacro{\angleQPerp}{\angleQ + 90}
      \global\let\angleQ\angleQ
      \global\let\angleQPerp\angleQPerp
    };

  \path[name path=sphereQ] (Q) circle (\RQ);
  \path[name path=sphereO] (O) circle (\RO);

  \path[name intersections={of=sphereQ and sphereO, by={P1, P2}}];
  
  \coordinate (M) at ($(P1)!0.5!(P2)$);


  \begin{scope}
    \clip (Q) circle (\RQ);
    \begin{scope}[shift={(Q)}, rotate=\angleERPerp]
      \clip (\RQ, 0) arc (0:180:\RQ\space and 0.65) -- (-\RQ, -3.5) -- (\RQ, -3.5) -- cycle;
      \fill[red!20, opacity=0.65] (0,0) circle (\RQ);
      \fill[pattern=north west lines, pattern color=red!45, opacity=0.7] (0,0) circle (\RQ);
    \end{scope}
  \end{scope}

  \begin{scope}
    \clip (Q) circle (\RQ);
    \begin{scope}[shift={(Q)}, rotate=\angleERPerp]
      \clip (-\RQ, 0) arc (180:360:{.8*\RQ\space+.55} and {.8*0.65}) -- (\RQ, -3.5) -- (-\RQ, -3.5) -- cycle;
      \fill[shade, shading=radial, inner color=red!40, outer color=red!80, opacity=0.6] (0,0) circle (\RQ);
      \fill[pattern=north west lines, pattern color=red!90!black, opacity=0.55] (0,0) circle (\RQ);
    \end{scope}
  \end{scope}

  \begin{scope}
    \clip (O) circle (\RO);
    \begin{scope}[shift={(O)}]

    \clip (-\RO, 0) arc (180:360:\RO\space and {.23*\RO}) -- (\RO, \RO) -- (-\RO, \RO) -- cycle;
      \fill[pattern=north west lines, pattern color=black!35, opacity=0.75] (0,0) circle (\RO);
    \end{scope}
  \end{scope}

  \draw[thick, black!85] (Q) circle (\RQ);
  \draw[thick, black!85] (O) circle (\RO);

  \draw[gray!70, thick] (O) ++(-\RO, 0) arc (180:360:\RO\space and 0.4);
  \draw[gray!70, dashed, thick] (O) ++(\RO, 0) arc (0:180:\RO\space and 0.4);

  \begin{scope}[shift={(Q)}, rotate=\angleERPerp]
    \draw[red!85!black, ultra thick] (-\RQ, 0) arc (180:360:{.8*\RQ\space+.55} and {.8*0.65});
    \draw[red!85!black, thick, dotted] (\RQ, 0) arc (0:180:\RQ\space and 0.65);
  \end{scope}

  \path let \p1 = ($(P1)-(M)$), \n1 = {veclen(\x1,\y1)} in
    \pgfextra{
      \pgfmathsetmacro{\radSigma}{\n1/1cm}
      \pgfmathsetmacro{\minorSigma}{0.42 * \radSigma}
      \global\let\radSigma\radSigma
      \global\let\minorSigma\minorSigma
    };

  \begin{scope}[shift={(M)}, rotate=\angleQPerp]
    \draw[blue!85!black, ultra thick] (\radSigma, 0) arc (0:-180:{.9*\radSigma\space+.18} and {.8*\minorSigma});
    \draw[blue!85!black, thick, dashed] (\radSigma, 0) arc (0:180:{.9*\radSigma\space+.15} and {.9*\minorSigma-.4});
  \end{scope}


  \fill[black] (Q) circle (2.2pt) node[below left=2pt] {$q$};
  \fill[black] (O) circle (2.2pt) node[below right=2pt] {$0$};
  \fill[black] ($(M)+(.4,-\RO+.15)$) circle (2.5pt) node[below right=3.5pt] {$w$};

  \draw[->, thick, >=Stealth] (O) -- ++(0, 1.85) node[above=2pt] {$e_n$};

  \begin{scope}[shift={(Q)}]
    \draw[->, red!90!black, ultra thick, >=Stealth] (0,0) -- (\angleER:\RQ) node[below right=1pt] {$e_R$};
  \end{scope}

  \node[blue!85!black, above=3pt] at (P1) {$\Sigma_R$};
  \node[font=\large] at ($(O) + (1.1, 1.1)$) {$H$};
  \node[red!85!black, font=\large] at ($(Q) + (-2.7, -2.2)$) {$H_R$};

  \node[below left] at ($(Q) +(-1,2) + (-130:\RQ)$) {$\partial B(q,R)$};
  \node[below right] at ($(O) + (-40:\RO)$) {$\partial B(0,1)$};

\end{tikzpicture}
        \caption{Case 3 example}
        \label{fig:case 3}
    \end{subfigure}
    \hfill 
    \begin{subfigure}[b]{0.48\textwidth}
        \centering
         \begin{tikzpicture}[>=Stealth, scale=.8]

  \coordinate (O) at (1.2, 0.35);
  \def\RO{1.85}

  \coordinate (Q) at (-1.3, -0.25);
  \def\RQ{2.75}

  \def\angleER{118}
  \pgfmathsetmacro{\angleERPerp}{\angleER + 90}

  \path let \p1 = ($(Q)-(O)$) in 
    \pgfextra{
      \pgfmathsetmacro{\angleQ}{atan2(\y1,\x1)} 
      \pgfmathsetmacro{\angleQPerp}{\angleQ + 90}
      \global\let\angleQ\angleQ
      \global\let\angleQPerp\angleQPerp
    };

  \path[name path=sphereQ] (Q) circle (\RQ);
  \path[name path=sphereO] (O) circle (\RO);

  \path[name intersections={of=sphereQ and sphereO, by={P1, P2}}];
  
  \coordinate (M) at ($(P1)!0.5!(P2)$);


  \begin{scope}
    \clip (Q) circle (\RQ);
    \begin{scope}[shift={(Q)}, rotate=\angleERPerp]
      \clip (\RQ, 0) arc (0:180:\RQ\space and 0.65) -- (-\RQ, -3.5) -- (\RQ, -3.5) -- cycle;
      \fill[shade, shading=radial, inner color=red!40, outer color=red!80, opacity=0.6] (0,0) circle (\RQ);
      \fill[pattern=north west lines, pattern color=red!90!black, opacity=0.55] (0,0) circle (\RQ);
    \end{scope}
  \end{scope}

  \begin{scope}
    \clip (Q) circle (\RQ);
    \begin{scope}[shift={(Q)}, rotate=\angleERPerp]
      \clip (-\RQ, 0) arc (180:360:{.8*\RQ\space+.55} and {.8*0.65}) -- (\RQ, -3.5) -- (-\RQ, -3.5) -- cycle;
      \fill[red!20, opacity=0.65] (0,0) circle (\RQ);
      \fill[pattern=north west lines, pattern color=red!45, opacity=0.7] (0,0) circle (\RQ);
    \end{scope}
  \end{scope}

  \begin{scope}
    \clip (O) circle (\RO);
    \begin{scope}[shift={(O)}]

    \clip (-\RO, 0) arc (180:360:\RO\space and {.23*\RO}) -- (\RO, \RO) -- (-\RO, \RO) -- cycle;
      \fill[pattern=north west lines, pattern color=black!35, opacity=0.75] (0,0) circle (\RO);
    \end{scope}
  \end{scope}

  \draw[thick, black!85] (Q) circle (\RQ);
  \draw[thick, black!85] (O) circle (\RO);

  \draw[gray!70, thick] (O) ++(-\RO, 0) arc (180:360:\RO\space and 0.4);
  \draw[gray!70, dashed, thick] (O) ++(\RO, 0) arc (0:180:\RO\space and 0.4);

  \begin{scope}[shift={(Q)}, rotate=\angleERPerp]
    \draw[red!85!black, thick, dotted] (-\RQ, 0) arc (180:360:{.8*\RQ\space+.55} and {.8*0.65});
    \draw[red!85!black, ultra thick] (\RQ, 0) arc (0:180:\RQ\space and 0.65);
  \end{scope}

  \path let \p1 = ($(P1)-(M)$), \n1 = {veclen(\x1,\y1)} in
    \pgfextra{
      \pgfmathsetmacro{\radSigma}{\n1/1cm}
      \pgfmathsetmacro{\minorSigma}{0.42 * \radSigma}
      \global\let\radSigma\radSigma
      \global\let\minorSigma\minorSigma
    };

  \begin{scope}[shift={(M)}, rotate=\angleQPerp]
    \draw[blue!85!black, ultra thick] (\radSigma, 0) arc (0:-180:{.9*\radSigma\space+.18} and {.8*\minorSigma});
    \draw[blue!85!black, thick, dashed] (\radSigma, 0) arc (0:180:{.9*\radSigma\space+.15} and {.9*\minorSigma-.4});
  \end{scope}


  \fill[black] (Q) circle (2.2pt) node[below left=2pt] {$q$};
  \fill[black] (O) circle (2.2pt) node[below right=2pt] {$0$};
  \fill[black] ($(M)+(.09,.70)$) circle (2.5pt) node[left=3.5pt] {$w$};

  \draw[->, thick, >=Stealth] (O) -- ++(0, 1.85) node[above=2pt] {$e_n$};

  \begin{scope}[shift={(Q)}]
    \draw[->, red!90!black, ultra thick, >=Stealth] (0,0) -- (\angleER:\RQ) node[above left=1pt] {$e_R$};
  \end{scope}

  \node[blue!85!black, above=3pt] at (P1) {$\Sigma_R$};
  \node[font=\large] at ($(O) + (1.1, 1.1)$) {$H$};
  \node[red!85!black, font=\large] at ($(Q) + (-2.7, -2.2)$) {$H_R$};

  \node[below left] at ($(Q) +(-1,2) + (-130:\RQ)$) {$\partial B(q,R)$};
  \node[below right] at ($(O) + (-40:\RO)$) {$\partial B(0,1)$};

\end{tikzpicture}
        \caption{Case 4 example}
        \label{fig:case 4}
    \end{subfigure}
    
    \caption{Examples for Cases 3 \& 4}
    \label{fig:cases34}
\end{figure}
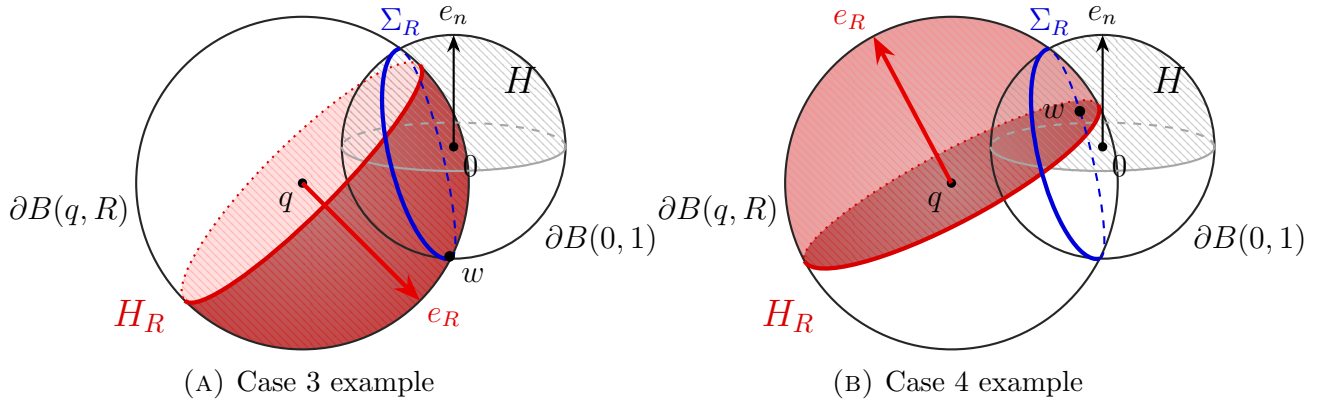

 \noindent\textbf{Case 3:} Suppose $e_R\cdot e_n \leq -\hat{C}\varepsilon^{4\gamma}$ for some $\hat{C} > 0$ large enough independent of $\varepsilon$. Recall also $|q\cdot e_R|\leq \varepsilon^{4\gamma}$.\\

  Let $d= -1$, in which case 
  $$w:=tq-\rho_Rv.$$ 

 We now show that $w\in H_R$. Recalling that $t-1<0$ and from \eqref{e: b approx e_r cdot e_n},
 \begin{align*}
     (w-q)\cdot e_R&= (t-1)q\cdot e_R-\rho_Rv\cdot e_R\\
     &\geq -\frac{R^2}{2}\varepsilon^{4\gamma} +C\rho_R\varepsilon^{4\gamma}\\
     &>\rho_R\varepsilon^{4\gamma},
 \end{align*}
 for $\varepsilon$ small enough and $\hat{C}$ large enough so that $\hat{C}\rho_R - \frac{R^2}{2} \geq \rho_R$.

 Moreover, $w\in \overline{H}^c$. Again from \eqref{e:q e_n} and \eqref{e: v e_n},
 \begin{align*}
     w\cdot e_n&=tq\cdot e_n-\rho_R v\cdot e_n\leq t\varepsilon^{16\gamma}-\rho_R(1-\frac{1}{2}\varepsilon^{32\gamma})
     \leq -\frac{1}{2}\rho_R,
 \end{align*}
 for $\varepsilon$ small enough. This concludes \textbf{Case 3}.

\noindent\textbf{Case 4:} Suppose $-\hat{C}\varepsilon^{4\gamma}\leq e_R\cdot e_n\leq 1-C\varepsilon^{2\gamma}$, and that $|q\cdot e_R|\leq \varepsilon^{4\gamma}$. Note that this is the final case by \eqref{e: e_q e_n}.\\

Recall $e_R = aq + bv + ce$ and $a^2 + b^2 + c^2 =1$. Since $|a| \leq \varepsilon^{4\gamma}$ and $|b| \leq 1-\frac{C}{2}\varepsilon^{2\gamma}$, for $\varepsilon > 0$ small enough, this implies that $|c| \geq \frac{\sqrt{C}}{2}\varepsilon^{\gamma}$. 

We can define 
$$
w:= tq + \varepsilon^{2\gamma}\rho_Rv - \mathrm{sign}(c)\sqrt{1-\varepsilon^{4\gamma}}\rho_Re.
$$

We first show that $w\in \overline{H}_R^c$. Indeed 

\begin{align*}
    (w-q)\cdot e_R =& (t-1)q\cdot e_R + \varepsilon^{2\gamma}\rho_Rv\cdot e_R - \mathrm{sign}(c)\sqrt{1-\varepsilon^{4\gamma}}\rho_Re\cdot e_R\\
    \leq& \frac{R^2}{2}\varepsilon^{4\gamma} + \rho_R\varepsilon^{2\gamma} - \frac{1}{2}\rho_R|c|\\
    \leq& -\rho_R\varepsilon^{\gamma},
\end{align*}
as long as $\varepsilon > 0$ is small enough. 

On the other hand we have $w\in H$ as 

\begin{align*}
    w\cdot e_n =& tq\cdot e_n + C\varepsilon^{2\gamma}\rho_Rv\cdot e_n-\mathrm{sign}(c)\sqrt{1-C^2\varepsilon^{\gamma}}\rho_Re\cdot e_n\\
    \geq& -\varepsilon^{16\gamma} +\frac{C}{2}\varepsilon^{2\gamma} \rho_R- \rho_R\varepsilon^{16\gamma}\\
    \geq& C\rho_R\varepsilon^{2\gamma},
\end{align*} as long as $\varepsilon > 0$ is small enough.

\medskip

\noindent\textbf{Step 3, the contradiction:} Now, we aim to apply Lemma \ref{lem: relative cap to cap}. Let $\tau\in (0,\frac{1}{4})$ be a small dimensional constant. Observe that 
$$
\mathcal{H}^{n-2}(\overline{\Omega}^{j}\cap B(w,\tau\varepsilon^{8\gamma}))\geq \left(\tau\varepsilon^{8\gamma}\right)^{n-2},
$$
for either $j=+$ or $j=-$. Suppose without loss of generality that $j=+$. We now apply Lemma \ref{lem: relative cap to cap} with $A:= \overline{\Omega}^+\cap B(w,\tau \varepsilon^{8\gamma})$ and $\Sigma=\Sigma_R\cap B(w,\frac{1}{2}\varepsilon^{8\gamma})$, a smooth $(n-2)$-hypersurface satisfying $d(\Sigma,\partial H)\gtrsim \varepsilon^{8\gamma}$ and $d(\Sigma,\partial H_R)\gtrsim \varepsilon^{8\gamma}$. Then for $x_0\in A$, 
$$
\mathcal{H}^{n-2}(A\cap B(x_0,2\tau\varepsilon^{8\gamma}))\geq\mathcal{H}^{n-2}(A)\geq \left(\tau\varepsilon^{8\gamma}\right)^{n-2},
$$
and so there exists a constant $C$, depending on dimension so that 
\begin{equation}\label{e: contradiction1}
    \varepsilon^{2(8\gamma)}\varepsilon^{8\gamma(n-1)}\lesssim_{\tau}C \mathrm{Cap}_{H}(A) \qquad \text{and} \qquad \varepsilon^{2(8\gamma)}\varepsilon^{8\gamma(n-1)}\lesssim_{\tau}C \mathrm{Cap}_{\,\overline{H_R}\,^c}(A) .
\end{equation}
But, from Lemma \ref{lem: no sets of pos capacity far from BAHP}, the fact that $8n\gamma \leq 1/12$ and the fact that $A\subset \overline{\Omega^+}\cap H_{\varepsilon^{8\gamma}}\subset \overline{\Omega^+}\cap H_{\varepsilon^{8n\gamma}}$, we have
\begin{equation}\label{e: contradiction2}
  \mathrm{Cap}_{\,\overline{H_R}\,^c}(A)\leq \mathrm{Cap}_{\overline{H_R}\,^c}(\overline{\Omega^+}\cap H_{\varepsilon^{8\gamma}})\leq\mathrm{Cap}_{\overline{H_R}\,^c}(\overline{\Omega^+}\cap H_{\varepsilon^{8n\gamma}}) <\varepsilon^{1/3}.  
\end{equation}

From \eqref{e: contradiction1} and \eqref{e: contradiction2} we obtain
\begin{equation*}
   \varepsilon^{1/25} \leq \varepsilon^{8\gamma(n+1)}\lesssim \mathrm{Cap}_{\,\overline{H_R}\,^c}(A) \lesssim \varepsilon^{1/3}, 
\end{equation*}
which is a contradiction. 

Thus there does not exist $y\in \Gamma$ with $d(y,\partial H)>\varepsilon^{\gamma}$. Thus
$$\begin{aligned} \{y\in \partial B(0,1)\mid (y-x) \cdot e_n \geq \varepsilon^{\gamma}\} \subset& \Omega^+, \qquad \text{and}\\
    \{y\in \partial B(0,1)\mid (y-x) \cdot e_n \leq -\varepsilon^{\gamma}\} \subset& \Omega^-, 
    \end{aligned}$$
as desired.
\end{proof}

In the specific case where $\lambda(x,r) \leq Cr^\alpha$ we have proven the following:

\begin{corollary}\label{c:ourinstance}
    Let $\Omega_1, \Omega_2 \subset \mathbb R^n$ and $C > 0, \alpha \in (0,1)$ and $r_0 > 0$ be such that for every $x\in \partial \Omega_1\cup \partial \Omega_2$ and $r < r_0$ we have $\lambda(x,r) \leq Cr^{\alpha}$.

    Then $\partial \Omega_1 = \partial \Omega_2 \equiv \Gamma$ and $$b\beta_\Gamma^s(x,r) \leq C'r^\gamma, \qquad \forall r < r_0/2$$ where $C',\gamma$ depend only on $r_0, C, \alpha, n$

    In fact more is true, borrowing the notation from Lemma \ref{lem:no sets of zero capactiy far from BAHP}, we have 

    \begin{equation}\label{e:reifdecayonspheres}\sup_{y\in \overline{\Omega^{\pm}} \cap \partial B(x,r)} \mathrm{dist}(y, H^{\pm}(x, r))\leq C'r^\gamma, \qquad \forall r < r_0/2, \end{equation}
    and 
    $$\begin{aligned} \{y\in \partial B(x,r)\mid (y-x) \cdot e_n \geq r^{\gamma+1}\} \subset& \Omega^+, \qquad \text{and}\\
    \{y\in \partial B(x_0,1)\mid (y-x) \cdot e_n \leq -r^{\gamma+1}\} \subset& \Omega^-.
    \end{aligned}$$
\end{corollary}
\begin{figure}
    \centering
\begin{tikzpicture}[scale=.6,every node/.style={minimum size=1cm}]

\def\R{4} 

\def\angEl{10} 
\def\angAz{-100} 
\def\angPhiOne{-110} 
\def\angPhiTwo{-45} 
\def\angBeta{30} 


\pgfmathsetmacro\H{\R*cos(\angEl)} 
\LongitudePlane[xzplane]{\angEl}{\angAz}
\LatitudePlane[equator]{\angEl}{0}
\fill[ball color=white!10] (0,0) circle (\R); 

\DrawLatitudeCircle[\R]{0} 


 \def\eps{1.2}
	\coordinate (D) at (-4,{.2*\eps});
	\def\d{0};

	\coordinate (E) at (-2,{-.6*\eps-.5});
	\def\e{30};
	\coordinate (F) at (-2.4,{.4*\eps-.3});
	\def\f{210};
    \coordinate (G0) at (-1.75,{-.15*\eps-.5});
	  \def\gz{91};
    \coordinate (G1) at (-1.04,{-.17*\eps-.5});
 \def\ge{310};


	\coordinate (G) at (-1,{-.2*\eps-.5});
	\def\g{-50};
    
	\coordinate (H) at (0,{0*\eps-.4});
	\def\h{30};
    \coordinate (H1) at (-.46,{-0.25*\eps-.4});
	\def\ha{35};
    \coordinate (H2) at (.44,{-0.24*\eps-.4});
	\def\hz{122};

    
	\coordinate (I) at (.8,{-.1*\eps-.8});
	\def\i{20};
    
	\coordinate (J) at (2.6,{.3*\eps-.4});
	\def\j{10};
	\coordinate (K) at (3.2,{0*\eps-.4});
	\def\k{-90};
\coordinate (K1) at (1.35,{-0.2*\eps-.4});
	\def\ka{35};

    
    \coordinate (L) at (3,{-.6*\eps-.4});
	\def\l{-90};
	\coordinate (M) at (3.88,{-.8*\eps});
	\def\m{15};
    
 \def\addsm{.9}
	\coordinate (T) at (-2,{-.6*\eps+\addsm});
	\def\t{40};
    
	\coordinate (U) at (-2.4,{.4*\eps+\addsm});
	\def\u{-50};
	\coordinate (U1) at (-1.93,{-.25*\eps+\addsm});
	\def\ua{105};
    \coordinate (U2) at (-1.85,{-.24*\eps+\addsm});
	\def\uz{121};


	\coordinate (V) at (0,{0*\eps-.5});
	\def\v{30};
    
	\coordinate (W) at (0,{0*\eps+\addsm});
	\def\w{30};
	\coordinate (W1) at (-0.16,{-0.18*\eps+\addsm});
	\def\wa{80};
    \coordinate (W2) at (0.6,{-0.18*\eps+\addsm});
	\def\wz{143};
    
	\coordinate (X) at (.8,{-.1*\eps+\addsm-.1});
	\def\x{20};
	\coordinate (Y) at (2.6,{.3*\eps+\addsm});
	\def\y{10};
\coordinate (X1) at (.85,{-.1*\eps+\addsm-.1});
	\def\xa{22};
\coordinate (Y1) at (3.23,{-0.4*\eps+\addsm});
	\def\ya{95};

    
	\coordinate (Z) at (3.2,{0*\eps+\addsm});
	\def\z{-90};

    \coordinate (A) at (-2,2);
	\def\a{45};
	\coordinate (B) at (-1.8,2.2);
	\def\b{30};
	\coordinate (C) at (-1.6,2.4);
	\def\c{10};
	\coordinate (N1) at (-1.4,2.3);
	\def\n{0};
	\coordinate (O1) at (-1,2.2);
	\def\o{-100};
	\coordinate (P) at (-1.5,2.1);
	\def\p{200};
    \coordinate (R) at (-1.9,1.8);
	\def\r{190};
    \coordinate (S1) at (-1.2,2.5);
	\def\s{-20};

\filldraw[nearly transparent, fill=frontshade] 
    (M) arc (-13.7:175:4) to
	(D) to [out=\d,in={\e-180}] 
	(E) to [out=\e,in={\f-180}] 
	(F) to [out=\f-180,in={\g-180}] 
	(G) to [out=\g,in={\h-180}] 
	(H) to [out=\h,in={\i-180}] 
	(I) to [out=\i,in={\j-180}] 
	(J) to [out=\j,in={\k-180}] 
	(K) to [out=\k,in={\l-180}] 
	(L) to [out=\l,in={\m-180}] 
    (M) -- cycle;

    \filldraw[nearly transparent, fill=backshade, draw=backshade] 
    (M) arc (-13.7:175:4) to
	(D) to [out=\d,in={\t-180}] 
	(T) to [out=\t,in={\u-180}] 
	(U) to [out=\u,in={\v-180}] 
	(V) to [out=\v,in={\w-180}] 
	(W) to [out=\w,in={\x-180}] 
	(X) to [out=\x,in={\y-180}] 
	(Y) to [out=\y,in={\z-180}] 
	(Z) to [out=\z,in={\m-180}] 
    (M) -- cycle;

 \foreach \longangthree in {-360,-358,...,-180}{
\DrawLongitudeCirclearcSouthHemispherebacksideONE[\R]{\longangthree}
}
    
 \draw[very thick,blue] 
	(D) to [out=\d,in={\e-180}] 
	(E) to [out=\e,in={\f-180}] 
	(F) to [out=\f-180,in={\g-180}] 
	(G) to [out=\g,in={\h-180}] 
	(H) to [out=\h,in={\i-180}] 
	(I) to [out=\i,in={\j-180}] 
	(J) to [out=\j,in={\k-180}] 
	(K) to [out=\k,in={\l-180}] 
	(L) to [out=\l,in={\m-180}] 
    (M);
  \draw[very thick,dashed,blue] 
	(D) to [out=\d,in={\t-180}] 
	(T) to [out=\t,in={\u-180}] 
	(U) to [out=\u,in={\v-180}] 
	(V) to [out=\v,in={\w-180}] 
	(W) to [out=\w,in={\x-180}] 
	(X) to [out=\x,in={\y-180}] 
	(Y) to [out=\y,in={\z-180}] 
	(Z) to [out=\z,in={\m-180}] 
    (M);

\coordinate[mark coordinate] (O) at (0,0);
\coordinate (N) at (0,{\H+.05});
\coordinate (S) at (0,{-\H-.05});

\node[below right=-.3pt] at (O) {$\mathbf{x}$};
\DrawLatitudeCirclethick[\R]{19} 
\DrawLatitudeCirclethickGREEN[\R]{-19} 

 \foreach \longang in {0,2,...,180}{
\DrawLongitudeCirclearcNorthHemisphereONE[\R]{\longang};
}

 \foreach \longangtwo in {-180,-178,...,0}{
\DrawLongitudeCirclearcSouthHemisphereONE[\R]{\longangtwo};
}

\end{tikzpicture}
    \caption{Consequences of Corollary \ref{c:ourinstance}}
    \label{fig:separationproperty}
\centering
\end{figure}
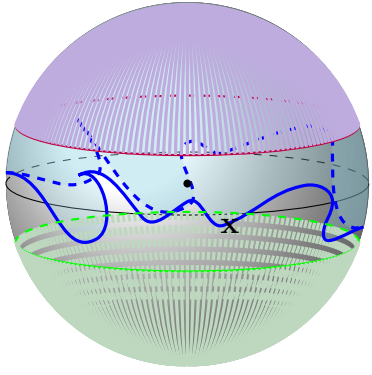

\section{The Geometric Argument}\label{s:geometric}
In this section we argue, purely geometrically, to show that decay on the spherical $\beta$-numbers implies decay on the classic bilateral $\beta$-numbers. We then put all the ingredients together to finish the proof of our main Theorem \ref{thrm:main theorem}.

\begin{lemma}\label{l:sphericaltofull}
Suppose $\Omega_i$, $i=1,2$, are open subsets of $\mathbb{R}^{n}$ and let $\Gamma = \partial \Omega_1\cup \partial \Omega_2$. Let $x_0 \in \Gamma, 10 \geq R > 0$ and $\varepsilon \in (0,\frac{1}{200})$ be such that for all $x\in \Gamma\cap B(x_0, 2R)$ and $r < R$ there exists a hyperplane $P = P(x, r)$, passing through $x$ with the following properties:
\begin{itemize}
    \item[(i)] 
    \begin{equation}\label{e: spherical beta quant decay}
        \sup_{y\in \Gamma \cap \partial B(x,r)} \mathrm{dist}(y, P\cap \partial B(x,r)) +\sup_{y\in P\cap \partial B(x,r)}\mathrm{dist}(y, \Gamma \cap \partial B(x,r))  \leq r \varepsilon,
    \end{equation}
    in particular, $b\beta_{\Gamma}^s(x,r) \leq \varepsilon$.

    \item[(ii)] If $\hat{n} = \hat{n}(x,r)$ is the unit normal to $P$ then $$\begin{aligned} \{y\in \partial B(x,r)\mid (y-x) \cdot \hat{n} \geq \frac{1}{4}r\} \subset& \Omega_i\\
    \{y\in \partial B(x,r)\mid (y-x) \cdot \hat{n} \leq -\frac{1}{4}r\} \subset& \Omega_j,
    \end{aligned}$$ for some assignment of $i,j\in \{1,2\}$ with $i\neq j$.  
\end{itemize}
   Then there exists a $C > 0$ depending only on the dimension, such that 
    \begin{equation}\label{e:bilat beta}
        b\beta_{\Gamma}(x_0,R)\leq M\varepsilon^{1/4}.
    \end{equation}

    In the special case where \eqref{e: spherical beta quant decay} is sharpened to $b\beta_\Gamma^s(x,r) \leq r^\alpha$ for all $r < R$ then we can conclude $b\beta_\Gamma(x_0,r) \leq Cr^\alpha$ for a dimensional constant $C > 0$ as long as $r \leq R < 1/2$. 
\end{lemma}
A set that satisfies property $(ii)$, we say satisfies the \textit{separation property} on spheres.

\begin{proof}
First we note that it is sufficient to prove \eqref{e:bilat beta} for $\varepsilon \leq \varepsilon_0$, where $\varepsilon_0$ is sufficiently small depending on dimension. Since $b\beta_{\Gamma}(x_0,R)\lesssim 1$, if $\varepsilon \geq \varepsilon_0$, then \eqref{e:bilat beta} is trivially true as long as $M$ is large enough depending on $\varepsilon_0$.  

 Let $y\in \Gamma\cap \partial B(x_0,R)$ so that $x_0\in \Gamma \cap \partial B(y,R)$. Then from \eqref{e: spherical beta quant decay}, $x_0\in N_{\varepsilon R}(P(y,R))$ where we use $N_\delta(A)$ to denote the $\delta$-neighborhood of $A$. 
 Now, for $s\leq R$  
$$
\Sigma_s:=\partial B(y,R)\cap \partial B(x_0,s)
$$
is an $(n-2)$-dimensional sphere of radius $r(\Sigma_s) :=\frac{1}{2}s\sqrt{4-(s/R)^2} \approx s$ with center $t_sx_0$ where $t_s=1-\frac{s^2}{2R^2}$. By translating or rotating if necessary, assume $y= 0$ and  $P(0,R)=\{x_n=0\}$. Assume further that $\varepsilon^{1/4}R\leq s$.

We first prove the following claim. \\
\noindent\textit{Claim 1:}
For every $q\in \Sigma_s \cap P(0,R)$ there exists a $p \in \Sigma_s\cap \Gamma$ such that $|p-q| < R\varepsilon$.\\

Indeed, let $A_1$ denote the minimal arc on $\Sigma_s \cap \{x_n\geq 0\}$ containing $q$ and  $t_sx_0+\rho_sv_n$, where $v_n\in \mathbb S^{n-1}$ is chosen so that $t_sx_0+\rho_sv_n\in \Sigma_s$ and $|v_n\cdot e_n|\geq 1-\varepsilon$. Such a $v_n$ exists because $(\Sigma_s-t_sx_0) \subset (\mathrm{span}(x_0))^\perp$ and $|x_0 \cdot e_n| < \varepsilon R$. 

We abuse terminology a bit and call $t_sx_0+\rho_sv_n$ the ``north pole'' of $\Sigma_s$. Similarly, let $A_2$ denote the minimal arc on $\Sigma_s \cap \{x_n\leq 0\}$ containing $q$ and $t_sx_0-\rho_sv_n$, i.e. the ``south pole'' of $\Sigma_s$, where $v_n$ as above.

From property (i) applied to $\partial B(0,R)$,
$$
A_1\cap \{x_n>R\varepsilon\}\cap \Gamma=\emptyset\qquad \text{and}\qquad A_2\cap \{x_n<-R\varepsilon\}\cap \Gamma=\emptyset. 
$$
Moreover, from property (ii) applied to $\partial B(0,R)$,
$$
A_1\cap \{x_n>\frac{R}{4}\}\subset \Omega_1\qquad \text{and}\qquad A_2\cap \{x_n<-\frac{R}{4}\}\subset \Omega_2, 
$$
up to relabeling $\Omega_1$ and $\Omega_2$. Thus, 
$$
A_1\cap \{x_n>R\varepsilon\}\subset \Omega_1\qquad \text{and}\qquad A_2\cap \{x_n<-R\varepsilon\}\subset \Omega_2, 
$$
which means that there exists $p\in \Gamma\cap \{-R\varepsilon\leq x_n\leq R\varepsilon\}\cap (A_1\cup A_2)$. Thus the proof of Claim 1 is complete.

We now show that the angle between $P(x_0,R)$ and $P(x_0,s)$ is small for all $\varepsilon^{1/4}R\leq s\leq R$. By the triangle inequality it is sufficient to show $P(0,R)+x_0$ and $P(x_0,s)$ have small angle for the whole range of $s$. Our strategy is as follows: we construct an auxiliary $(n-1)$-plane, $P_1$, which is determined by a set of $(n-2)$- linearly independent points on $\Sigma_s\cap \Gamma$. Since $(\Sigma_s-t_sx_0)\subset (\mathrm{span}(x_0))^\perp$, this collection of points together with $x_0$ determines an $(n-1)$-plane, which we call $P_1$, that has small angle with both $P(0,R)+x_0$ and $P(x_0,s)$.

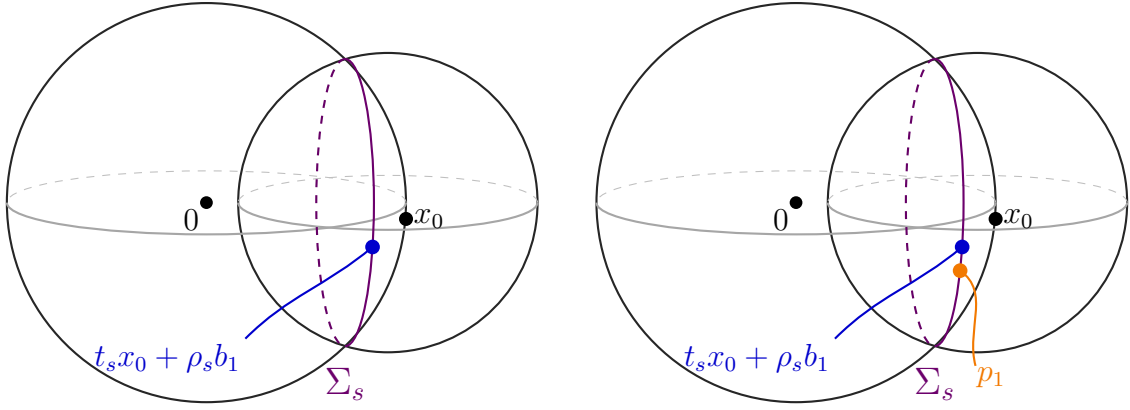
\begin{figure}
    \centering
\begin{tikzpicture}[>=Stealth, scale=1.2]

  \def\Rleft{2.2}      
  \def\Rright{1.65}    
  \def\distCenters{2.0}

  \begin{scope}[local bounding box=panel1]
    \coordinate (O) at (-1.3, 0);                   
    \coordinate (O2) at ($(O) + (\distCenters,0)$); 

    \path[name path=sphereL1] (O) circle (\Rleft);
    \path[name path=sphereR1] (O2) circle (\Rright);

    \draw[thick, black!85] (O) circle (\Rleft);
    \draw[thick, black!85] (O2) circle (\Rright);

    \draw[gray!70, thick] (O) ++(-\Rleft, 0) arc (180:360:\Rleft\space and 0.35);
    \draw[gray!50, dashed] (O) ++(\Rleft, 0) arc (0:180:\Rleft\space and 0.35);

    \draw[gray!70, thick] (O2) ++(-\Rright, 0) arc (180:360:\Rright\space and 0.3);
    \draw[gray!50, dashed] (O2) ++(\Rright, 0) arc (0:180:\Rright\space and 0.3);

    \path[name intersections={of=sphereL1 and sphereR1, by={Ptop, Pbot}}];
    \coordinate (M) at ($(Ptop)!0.5!(Pbot)$);

    \path let \p1 = ($(Ptop)-(M)$), \n1 = {veclen(\x1,\y1)} in
      \pgfextra{
        \pgfmathsetmacro{\radSigma}{\n1/1cm}
        \pgfmathsetmacro{\minorSigma}{0.28 / sqrt(1 - (-0.75/\radSigma)^2)}
        \global\let\radSigma\radSigma
        \global\let\minorSigma\minorSigma
      };

    \draw[violet!85!black, thick] (M) ++(0, -\radSigma) arc (-90:90:\minorSigma\space and \radSigma);
    \draw[violet!80!black, dashed, thick] (M) ++(0, \radSigma) arc (90:270:\minorSigma\space and \radSigma);

    \coordinate (X0) at ($(O2) + (0.2, -0.18)$); 
    \coordinate (B1) at ($(M) + ({\minorSigma*cos(15)}, {\radSigma*sin(15)})$);   
    \coordinate (TS) at ($(M) + ({\minorSigma*cos(-18)}, {\radSigma*sin(-18)})$); 


    \fill[black] (O) circle (2pt) node[below left=1pt] {$0$};
    \fill[black] (X0) circle (2.2pt) node[right=1pt] {$x_0$};


    \fill[blue!80!black] (TS) circle (2.3pt);
    \draw[blue!80!black, thick] ($(M) + (-1.1, -1.5)$) node[below left] {$t_s x_0 + \rho_s b_1$}
      to[out=45, in=-140] (TS);

    \node[violet!85!black, font=\large] at ($(M) + (0, -2)$) {$\Sigma_s$};

  \end{scope}

  \begin{scope}[shift={(6.5,0)}, local bounding box=panel2]
    \coordinate (O) at (-1.3, 0);                   
    \coordinate (O2) at ($(O) + (\distCenters,0)$); 

    \path[name path=sphereL2] (O) circle (\Rleft);
    \path[name path=sphereR2] (O2) circle (\Rright);

    \draw[thick, black!85] (O) circle (\Rleft);
    \draw[thick, black!85] (O2) circle (\Rright);

    \draw[gray!70, thick] (O) ++(-\Rleft, 0) arc (180:360:\Rleft\space and 0.35);
    \draw[gray!50, dashed] (O) ++(\Rleft, 0) arc (0:180:\Rleft\space and 0.35);

    \draw[gray!70, thick] (O2) ++(-\Rright, 0) arc (180:360:\Rright\space and 0.3);
    \draw[gray!50, dashed] (O2) ++(\Rright, 0) arc (0:180:\Rright\space and 0.3);

    \path[name intersections={of=sphereL2 and sphereR2, by={Ptop2, Pbot2}}];
    \coordinate (M) at ($(Ptop2)!0.5!(Pbot2)$);

    \draw[violet!85!black, thick] (M) ++(0, -\radSigma) arc (-90:90:\minorSigma\space and \radSigma);
    \draw[violet!80!black, dashed, thick] (M) ++(0, \radSigma) arc (90:270:\minorSigma\space and \radSigma);

    \coordinate (X0) at ($(O2) + (0.2, -0.18)$); 
    \coordinate (B1) at ($(M) + ({\minorSigma*cos(15)}, {\radSigma*sin(15)})$);   
    \coordinate (TS) at ($(M) + ({\minorSigma*cos(-18)}, {\radSigma*sin(-18)})$); 
    \coordinate (P1) at ($(M) + (0.28, -0.75)$);                                 

    \fill[black] (O) circle (2pt) node[below left=1pt] {$0$};
    \fill[black] (X0) circle (2.2pt) node[right=1pt] {$x_0$};

    \fill[blue!80!black] (TS) circle (2.3pt);
    \draw[blue!80!black, thick] ($(M) + (-1.1, -1.5)$) node[below left] {$t_s x_0 + \rho_s b_1$}
      to[out=45, in=-140] (TS);

    \fill[orange!95!black] (P1) circle (2.3pt);
    \draw[orange!95!black, thick] ($(M) + (0.45, -1.8)$) node[below right=-2pt] {$p_1$}
      to[out=105, in=-35] (P1);

    \node[violet!85!black, font=\large] at ($(M) + (0, -2)$) {$\Sigma_s$};

  \end{scope}

\end{tikzpicture}
    \caption{$\Sigma_s$ and construction of points in the case $n=3$}
    \label{fig:n=3 example}
\end{figure}

Let us begin by constructing the set of such points. Observe that $(P(0,R)\cap \Sigma_s)-t_sx_0\subset (\mathrm{span}(e_n,x_0))^{\perp}$. Let $b_1, \ldots, b_{n-2}$ be an orthonormal basis of $\mathrm{span}(x_0, e_n)^{\perp}$. 
Then the points $\{t_sx_0+\rho_sb_i\}_{i=1}^{n-2}\subset \Sigma_s$. From property (i) applied to $\partial B(0,R)$, 
$$|e_n \cdot (t_sx_0 + \rho_sb_i)| = |t_s||e_n \cdot x_0| < \varepsilon R,$$
or in other words 
\begin{equation*}
    d(t_sx_0+\rho_sb_i, P(0,R))\leq \varepsilon R \qquad \text{for}\qquad 1\leq i\leq n-2.
\end{equation*}

Then by \textit{Claim 1} and the triangle inequality, there exist points, $\{p_i\}_{i=1}^{n-2}\subset \Sigma_s\cap \Gamma$ such that 
\begin{equation}\label{e:lin indep}
    |p_i - (t_sx_0 + \rho_sb_i)| \leq 2\varepsilon R.
\end{equation}

In particular, this means that $d(p_i,P(0,R))\leq 3\varepsilon R$ for $1\leq i\leq n-2$. Furthermore, the points $p_i - t_sx_0$ are quantitatively linearly independent in the sense that 
$$
\begin{aligned} 
\left|(p_i -t_sx_0)\cdot (p_j - t_sx_0)\right| 
&\leq \left|(p_i - (t_sx_0 + \rho_s b_i))\cdot (p_j  -(t_s x_0 + \rho_s b_j))\right| + \left|\rho_sb_i \cdot \rho_s b_j\right|\\
&\hspace{.3cm}+\left|(p_i - (t_sx_0 + \rho_sb_i))\cdot \rho_sb_j\right| +\left|(p_j - (t_sx_0 + \rho_sb_j))\cdot \rho_sb_i \right|\\
&\leq 4\varepsilon^2 R^2 + 4\varepsilon R\rho_s\\
&\approx8\varepsilon R\rho_s\\
&\leq 8 \varepsilon^{3/4}  \rho_s^2,
\end{aligned}
$$
where we have used \eqref{e:lin indep}, and the facts that $\rho_s\approx s$ and $\varepsilon R\leq \varepsilon^{1/4}R\leq s\leq R$. Thus, 
\begin{equation}\label{e: p_i p_j}
    \left|\frac{(p_i -t_sx_0)}{\rho_s}\cdot \frac{(p_j -t_sx_0)}{\rho_s}\right|\lesssim \varepsilon^{3/4}.
\end{equation}
Moreover, from \eqref{e:lin indep}, the definition of the $b_i's$ and the bounds on $s$ we have
\begin{align*}
    \left|x_0\cdot (p_i-t_sx_0)\right|&\leq |x_0||p_i-(t_sx_0+\rho_s b_i)|+\rho_s|x_0\cdot b_i|\leq |x_0|2\varepsilon R\leq 2|x_0|\varepsilon^{3/4}\rho_s.
\end{align*}
Thus,
\begin{equation}\label{e: x_0 p_i}
    \left|\frac{x_0}{|x_0|}\cdot \frac{(p_i-t_sx_0)}{\rho_s}\right|\lesssim \varepsilon^{3/4},
\end{equation}
 and from \eqref{e: p_i p_j} and \eqref{e: x_0 p_i} we can conclude that the set of point $\{x_0/|x_0|, p_1-t_sx_0, \ldots, p_{n-2}-t_sx_0\}$ are quantitatively linearly independent.

Let $P_1:= \mathrm{span}(x_0, p_1 - t_sx_0, \hdots, p_{n-2}-t_sx_0)=\mathrm{span}(x_0,p_1,\hdots, p_{n-2}) $ and let $P(x_0, s)$ be a plane which achieves the infimum for $b\beta_\Gamma^s(x_0,s)$. In particular $P(x_0,s)$ is a $(n-1)$-affine plane which contains $x_0$ and satisfies 
\begin{equation}\label{e: p_i's close}
    d(p_i,P(x_0,s))\leq s\varepsilon \qquad \text{for all}\qquad 1\leq i\leq n-2, 
\end{equation}
and if $p_1^- \in \Sigma_s\cap \Gamma$ is a point which satisfies $|p_1^- - (t_sx_0 -\rho_sb_1)| \leq 2\varepsilon R$, then also 
\begin{equation}\label{e: p- close}
    d(p_1^-, P(x_0, s)) \leq s\varepsilon.
\end{equation}

We now aim to prove that the auxiliary plane $P_1$ has small angle with both $P(x_0,s)-x_0$ and $P(0,R)$. Since $x_0,p_i\in \partial B(0,R)\cap \Gamma\subset N_{\varepsilon R}(P(0,R))$ for all $1\leq i\leq n-2$ from \eqref{e: spherical beta quant decay}, it follows immediately that 
\begin{equation}\label{e:small angle planes}
    \measuredangle(P(0,R), P_1)\leq C\varepsilon, 
\end{equation}
for some constant $C:=C(n)>0$, where $\measuredangle(P(0,R),P_1)$ denotes the angle between $P(0,R)$ and $P_1$. We now prove that $P_1$ has small angle with $P(x_0,s)-x_0$, which is the content of \textit{Claim 2.}\\

\noindent\textit{Claim 2:} $\measuredangle \left(P(x_0,s)-x_0,P_1\right)\leq C \varepsilon^{1/4}$ for some constant $C$ depending only on dimension. \\

To prove \textit{Claim 2}, it is sufficient to show that there exists a unit vector $\hat{e}\in P(x_0, s)- x_0$ such that $\hat{e}\cdot  \frac{x_0}{|x_0|} \geq  1-C\varepsilon^{1/2}$ for some constant $C:=C(n) > 0$. Indeed, if such a vector exists, then 
\begin{equation*}
    \mathrm{dist}(\frac{x_0}{|x_0|}, P(x_0,s)- x_0) \leq |\frac{x_0}{|x_0|}-\hat{e}| < C\varepsilon^{1/4},
\end{equation*}
and 
\begin{align*}
     \mathrm{dist}\left(\frac{p_i - t_sx_0}{\rho_s}, P(x_0, s)-x_0\right) 
     &\leq \frac{1}{\rho_s}\mathrm{dist}\left(p_i-x_0, P(x_0, s)-x_0\right)+\frac{1}{\rho_s}\mathrm{dist}\left((1-t_s)x_0, P(x_0,s)-x_0\right)\\
     &\leq \frac{1}{\rho_s}\varepsilon s+\frac{R(1-t_s)}{\rho_s}\left|\frac{x_0}{|x_0|}-\hat{e}\right|\\
     &\leq C\varepsilon +C\varepsilon^{1/4}\\
     &\approx \varepsilon^{1/4},
\end{align*}
from \eqref{e: p_i's close}, and since $|x_0|=R$, $\rho_s\approx s$, and $1-t_s\approx s^2/R^2$. Then from the fact that $\{\frac{x_0}{|x_0|}, \frac{1}{\rho_s}(p_1 - t_sx_0), \hdots, \frac{1}{\rho_s}(p_{n-2}-t_sx_0)\}$ are quantitatively linearly independent and all $C\varepsilon^{1/4}$-close to $P(x_0, s)-x_0$, we can conclude \begin{equation}\label{e:firstanglesmall}
\measuredangle(P(x_0,s)-x_0,P_1)\leq C\varepsilon^{1/4},
\end{equation} 
where $C > 0$ depends only on the dimension.

So we now prove the existence of such a vector, $\hat{e}$, in order to complete the proof of \textit{Claim 2}. It follows from \eqref{e: p_i's close} and \eqref{e: p- close} that there exists a vector $\tilde{e} \in P(x_0,s) - x_0$ such that  
\begin{align*}
    \left|\tilde{e}-\frac{1}{s}((p_1-x_0) + (p_1^--x_0)) \right|\leq 2\varepsilon.
\end{align*}
It then follows, also from \eqref{e:lin indep}, that 
\begin{align*}
    \left|\tilde{e} + \frac{2(1-t_s)}{s}x_0\right|
    &\leq \left|\tilde{e}-\frac{1}{s}(p_1-x_0+p_1^--x_0)\right| + \frac{1}{s}\left|p_1-t_sx_0+p_1^--t_sx_0 \right|
    \leq C\varepsilon^{3/4},
\end{align*}
for all $R\varepsilon^{1/4}\leq s\leq R$. Moreover, observing that $|\frac{2(1-t_s)}{s}x_0| = \frac{s}{R}$, let $e^{\prime}:=\frac{-R}{s}\tilde{e}$, in which case $e^{\prime}\in P(x_0, s)- x_0$ and, 
$$
|e' - \frac{x_0}{|x_0|}| \leq C\frac{R}{s}\varepsilon^{3/4}\leq C \varepsilon^{1/2}.
$$ 
Taking $\hat{e} = e'/|e'|$ completes the proof of \textit{Claim 2}.

It then follows from \eqref{e:small angle planes} and \textit{Claim 2}, we have
 $$ 
 \measuredangle(P(0,R), P(x_0,s)-x_0) \leq  \measuredangle(P(0,R), P_1) + \measuredangle(P_1, P(x_0,s)-x_0) \leq C\varepsilon^{1/4},
 $$ 
 for all $R\varepsilon^{1/4}\leq s\leq  R$, and thus, 
 \begin{equation}\label{e: final angle}
       \measuredangle(P(x_0,R), P(x_0,s))\leq \measuredangle(P(x_0,R), P(0,R)+x_0)+\measuredangle(P(0,R)+x_0,P(x_0,s))\leq C\varepsilon^{1/4}, 
 \end{equation}
 for all $R\varepsilon^{1/4}\leq s\leq  R$.

\begin{figure}
    \centering
    \begin{tikzpicture}[>=Stealth, scale=1.2]

  \def\Rleft{2.2}      
  \def\Rright{1.65}    
  \def\distCenters{2.0}

  \coordinate (O) at (-1.3, 0);                   
  \coordinate (O2) at ($(O) + (\distCenters,0)$); 

  \path[name path=sphereL] (O) circle (\Rleft);
  \path[name path=sphereR] (O2) circle (\Rright);

  \coordinate (P1_BL) at (-3.8, -0.25);
  \coordinate (P1_TL) at (-2.9, 0.70);
  \coordinate (P1_TR) at (3.1, 0.45);
  \coordinate (P1_BR) at (3.0, -0.85);

  \coordinate (Px_BL) at (-3.5, -1.05);
  \coordinate (Px_TL) at (-3.4, -0.05);
  \coordinate (Px_TR) at (3.7, 0.8);
  \coordinate (Px_BR) at (3.6, -0.20);

  \draw[gray!70, thick] (Px_BL) -- (Px_TL) -- (Px_TR) -- (Px_BR) -- cycle;
  \draw[orange!90!black, thick] (P1_BL) -- (P1_TL) -- (P1_TR) -- (P1_BR) -- cycle;

  \draw[thick, black!85] (O) circle (\Rleft);
  \draw[thick, black!85] (O2) circle (\Rright);

  \draw[gray!70, thick] (O) ++(-\Rleft, 0) arc (180:360:\Rleft\space and 0.35);
  \draw[gray!50, dashed] (O) ++(\Rleft, 0) arc (0:180:\Rleft\space and 0.35);

  \draw[gray!70, thick] (O2) ++(-\Rright, 0) arc (180:360:\Rright\space and 0.30);
  \draw[gray!50, dashed] (O2) ++(\Rright, 0) arc (0:180:\Rright\space and 0.30);

  \path[name intersections={of=sphereL and sphereR, by={Ptop, Pbot}}];
  \coordinate (M) at ($(Ptop)!0.5!(Pbot)$);

  \path let \p1 = ($(Ptop)-(M)$), \n1 = {veclen(\x1,\y1)} in
    \pgfextra{
      \pgfmathsetmacro{\radSigma}{\n1/1cm}
      \pgfmathsetmacro{\minorSigma}{0.28 / sqrt(1 - (-0.75/\radSigma)^2)}
      \global\let\radSigma\radSigma
      \global\let\minorSigma\minorSigma
    };

  \draw[violet!85!black, thick] (M) ++(0, -\radSigma) arc (-90:90:\minorSigma\space and \radSigma);
  \draw[violet!80!black, dashed, thick] (M) ++(0, \radSigma) arc (90:270:\minorSigma\space and \radSigma);

  \coordinate (X0) at ($(O2) + (0.2, -0.18)$);                                
  \coordinate (P1_dot) at ($(M) + (0.28, -0.75)$);                             
  \coordinate (P_green) at ($(M) + ({\minorSigma*cos(15)}, {\radSigma*sin(15)})$); 

  \fill[black] (O) circle (2pt) node[below left=1pt] {$0$};
  \fill[black] (X0) circle (2pt) node[right=2pt] {$x_0$};
  \fill[orange!95!black] (P1_dot) circle (2.2pt);

  \coordinate (En_base) at ($(O) + (0.4, -0.05)$);


  \node[orange!90!black, font=\Large] at ($(P1_TR) + (0.2, -0.2)$) {$P_1$};
  \node[black!85, font=\Large] at ($(Px_TR) + (0.7, 0.1)$) {$P(x_0,s)$};
  \node[violet!85!black, font=\Large] at ($(M) + (0, -\radSigma - 0.5)$) {$\Sigma_s$};

  \draw[orange!90!black, thick] ($(M) + (0.5, -1.8)$) node[below right=-2pt] {$p_1$}
    to[out=105, in=-35] (P1_dot);

\end{tikzpicture}
    \caption{Auxiliary plane $P_1$ in case $n=3$}
    \label{fig:close planes}
\end{figure}
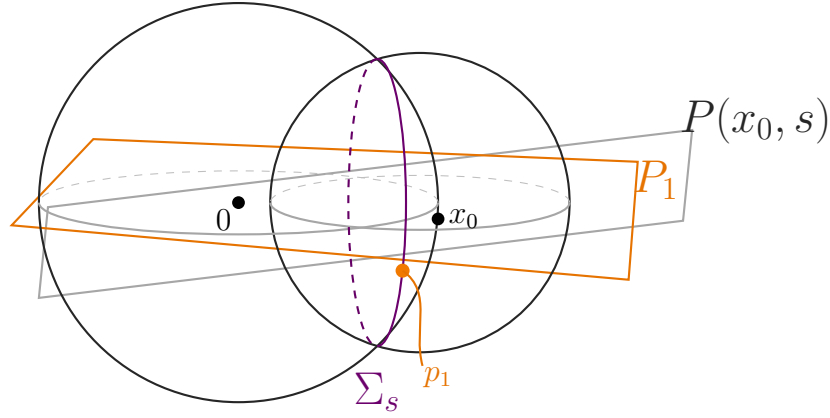

We now complete the proof of the lemma. Consider the plane $P(x_0,R)$. Then for any $x\in B(x_0,\varepsilon^{1/4}R)$, it follows immediately that 
$$
\mathrm{dist}(x, P(x_0,R)\cap B(x_0,R))\leq \varepsilon^{1/4}R\qquad \text{and} \qquad \mathrm{dist}(x, \Gamma\cap B(x_0,R))\leq \varepsilon^{1/4}R.
$$
So let $x\in B(x_0,R)\setminus B(x_0,R\varepsilon^{1/4})$. Then $x\in \partial B(x_0,s)$ for some $\varepsilon^{1/4}R\leq s\leq R$. 

If $x\in \Gamma$ then from \eqref{e: spherical beta quant decay} and \eqref{e: final angle}, we have
\begin{equation}\label{e:pt to plane}
   d(x,P(x_0,R)\cap B(x_0,R))\leq d(x,P(x_0,s))+s\measuredangle(P(x_0,s),P(x_0,R))\leq C\varepsilon^{1/4} s\leq C\varepsilon^{1/4} R. 
\end{equation}

If instead $x \in P(x_0,R))$, then from \eqref{e: spherical beta quant decay} and \eqref{e: final angle},
\begin{equation}\label{e: pt to boundary}
    d(x,\Gamma\cap B(x_0,R))\leq s\measuredangle(P(x_0,s),P(x_0,R))+\varepsilon s\leq C\varepsilon^{1/4} s\leq C\varepsilon^{1/4} R.
\end{equation}

Thus, \eqref{e:bilat beta} follows.

For the special case when $\varepsilon=r^{\alpha}$ in \eqref{e: spherical beta quant decay} (i.e. when $b\beta_\Gamma^s(x,r) \leq r^\alpha$), the argument in \textit{Claim 2} yields
$$
\measuredangle(P(x_0,s)-x_0, P_1)\leq C R^{\alpha} \qquad \text{for all} \qquad \frac{R}{2}\leq s\leq R,
$$
which allows us to conclude that
$$
\measuredangle(P(x_0,R), P(x_0,s)) \leq CR^\alpha, \qquad \text{for all} \qquad  \frac{R}{2}\leq s\leq R.
$$ 
Again, consider the plane $P(x_0,R)$. Observe that for any $0\leq s\leq R$, there exists $K_0\in \mathbb{N}$ such that $\frac{R}{2^{K_0 +1}}<s\leq \frac{R}{2^{K_0}}$, and then
\begin{align*}
    \measuredangle(P(x_0,R), P(x_0,s))
    &\leq \measuredangle\left(P\left(x_0,\frac{R}{2^{K_0}}\right),P(x_0,s)\right)+\sum_{k=0}^{K_0} \measuredangle\left(P\left(x_0,\frac{R}{2^k}\right),P\left(x_0,\frac{R}{2^{k+1}}\right)\right)\\
    &\leq C \left(\frac{R}{2^{K_0}}\right)^{\alpha} +CR^{\alpha}\sum_{k=0}^{K_0}2^{-k\alpha}\\
    &\leq C R^{\alpha},
\end{align*}
where $C$ depends only on dimension (and not on $K_0$). Then as in \eqref{e:pt to plane} and \eqref{e: pt to boundary}, we obtain 
$$
\sup_{x\in \Gamma \cap B(x_0,R)}\mathrm{dist}(x, P(x_0,R)\cap B(x_0,R))\leq CR^{\alpha+1},
$$ 
and
$$
\sup_{x\in P(x_0,R)\cap B(x_0,R)}\mathrm{dist}(x, \Gamma\cap B(x_0,R))\leq CR^{\alpha+1},
$$
which gives $b\beta_\Gamma(x_0,R)\leq CR^{\alpha}$, which is the desired bound. 
\end{proof}

We our now ready to put all the ingredients together and prove our main theorem:

\begin{proof}[Proof of Theorem \ref{thrm:main theorem}]
    It follows from Lemma \ref{lem:existence of corkscrews} that $\Omega_2=\mathbb{R}^n\setminus\overline{\Omega_1}$ and $\Omega_1$ is a two-sided corkscrew open set. Let us rename $\Omega^+:=\Omega_1$ and $\Omega^-:=\Omega_2$. The result now follows from Corollary \ref{c:ourinstance}, Lemma \ref{l:sphericaltofull}, and Theorem \ref{t:DKT}.
\end{proof}

\bibliographystyle{alpha}
\bibdata{references}
\bibliography{references}
 \end{document}